\documentclass[11pt]{article}

\usepackage{cancel,amsthm,amsmath,times,amsopn,amsfonts,amssymb,epsfig,anysize}
\usepackage{graphicx}
\usepackage[all]{xy}
\usepackage{pifont,txfonts,graphicx,geometry}
\usepackage[square,numbers]{natbib}
\usepackage{pb-diagram}

\usepackage{latexsym,graphics,color}
\usepackage[usenames,dvipsnames]{xcolor}
\usepackage{algorithmic}

\newtheorem{theorem}{Theorem}
\newtheorem{property}[theorem]{Property}
\newtheorem{corollary}[theorem]{Corollary}
\newtheorem{lemma}[theorem]{Lemma}

\newtheorem{proposition}[theorem]{Proposition}
\newtheorem{definition}[theorem]{Definition}
\newtheorem{remark}[theorem]{Remark}
\newtheorem{example}[theorem]{Example}
\newtheorem{conjecture}[theorem]{Conjecture}

\newcommand{\mfc}{{\mathfrak c}}
\newcommand{\mfa}{{\mathfrak a}}

\def \sh {\text{sh}}
\def \tS {\tilde S^0_n}

\def \kcocharge {k\text{-cocharge}}
\def \ncocharge {n\text{-cocharge}}
\def \ZZ {\mathbb Z}
\def \Z {\mathbb Z}
\def \CC {{\mathbb C}}
\def \Co {{\mathcal C^n}}

\def\kbnd {\mathfrak p}
\def\core {\mathfrak c}

\def \core {{\mathfrak c}}

\def\Gr {\text{Gr}}

\def\weight{ {\rm {weight}}}
\def\cocharge{ {\rm {cocharge}}}

\edef\savecatcodeat{\the\catcode`@}
\catcode`\@=11

\def\@IFNEXTCHAR#1#2#3{\let\@tempe #1\def\@tempa{#2}\def\@tempb{#3}\futurelet
    \@tempc\@IFNCH}
\def\@IFNCH{\ifx \@tempc \@sptoken \let\@tempd\@xifnch
      \else \ifx \@tempc \@tempe\let\@tempd\@tempa\else\let\@tempd\@tempb\fi
      \fi \@tempd}

\def\tb@ifSpecChars#1#2{#1}
\def\tb@ifNoSpecChars#1#2{#2}

\newcount\tb@rpN

\def\tableau{%
  \bgroup
  \@ifstar{\let\Tif\tb@ifNoSpecChars\tb@tableauB}
          {\let\Tif\tb@ifSpecChars\tb@tableauB}}
\def\tb@tableauB{%
  \@IFNEXTCHAR[{\tb@tableauC}{\tb@tableauC[]}}     

\def\tb@tableauC[#1]{\hbox\bgroup%
    \let\\=\cr
    \def\bl{\global\let\tbcellF\tb@cellNF}%
    \def\tf{\global\let\tbcellF\tb@cellH}
    \dimen2=\ht\strutbox \advance\dimen2 by\dp\strutbox%
    \ifx\baselinestretch\undefined\relax%
    \else%
       \dimen0=100sp \dimen0=\baselinestretch\dimen0%
       \dimen2=100\dimen2 \divide\dimen2 by\dimen0%
    \fi%
    \let\tpos\tb@vcenter
    \tb@initYoung
    \tb@options#1\eoo
    \let\arrow\tb@arrow%
    \dimen0=\Tscale\dimen2%
    \dimen1=\dimen0 \advance\dimen1 by \tb@fframe%
    \lineskip=0pt\baselineskip=0pt
    \def\tb@nothing{}%
    \def\endcellno{$\rss\egroup\bss\egroup}
    \def\endcell{\endcellno\kern-\dimen0}
    \def\begincell{\vbox to\dimen0\bgroup\vss\hbox to\dimen0\bgroup\hss$}%
    \let\overlay\tb@overlay%
    \let\fl\tb@fl%
    \let\fr\tb@fr%
    \let\lss\hss\let\rss\hss\let\tss\vss\let\bss\vss
    \def\mkcell##1{
        \let\tbcellF\tb@cellD
        \def\tb@cellarg{##1}
        \ifx\tb@cellarg\tb@nothing\let\tb@cellarg\tb@cellE\fi%
	        \begincell\tb@cellarg\endcellno
	        \tbcellF
    }%
    \let\savecellF\tbcellF
     \Tif{\catcode`,=4\catcode`|=\active}{}\tb@tableauD}%

\let\tb@savetableauD\tableauD
{
    \catcode`|=\active \catcode`*=\active \catcode`~=\active%
    \catcode`@=\active
\gdef\tableauD#1{%
  \Tif{
    \mathcode`|="8000 \mathcode`*="8000%
    \mathcode`~="8000 \mathcode`@="8000%
    \def@{\bullet}%
    \let|\cr
    \let*\tf
    \let~\sk
  }{}%
  \tpos{\tabskip=0pt\halign{&\mkcell{##}\cr#1\crcr}}%
  \global\let\tbcellF\savecellF
  \egroup
  \egroup}
}
\let\tb@tableauD\tableauD
\let\tableauD\tb@savetableauD
\let\tb@savetableauD\undefined

\def\tb@options#1{\ifx#1\eoo\relax\else\tb@option#1\expandafter\tb@options\fi}

\def\tb@option#1{%
  \if#1t\let\tpos\tb@vtop\fi
  \if#1c\let\tpos\tb@vcenter\fi
  \if#1b\let\tpos\vbox\fi
  \if#1F\tb@initFerrers\fi
  \if#1Y\tb@initYoung\fi
  \if#1E\tb@initEmpty\fi
  \if#1s\tb@initSmall\fi
  \if#1m\tb@initMedium\fi
  \if#1l\tb@initLarge\fi
  \if#1p\tb@initPartition\fi
  \if#1a\tb@initArrow\fi
}

\def\tb@vcenter#1{\ifmmode\vcenter{#1}\else$\vcenter{#1}$\fi}

\def\tb@vtop#1{\hbox{\raise\ht\strutbox\hbox{\lower\dimen0\vtop{#1}}}}

\def\tb@initPartition{\def\Tscale{.3}}
\def\tb@initSmall{\def\Tscale{1}}
\def\tb@initMedium{\def\Tscale{2}}
\def\tb@initLarge{\def\Tscale{3}}

\def\tb@initArrow{\dimen2=1.25em}

\def\tb@initYoung{%
  \def\tb@cellE{}
  \let\tb@cellD\tb@cellN
  \def\sk{\global\let\tbcellF\tb@cellNF}}
\def\tb@initFerrers{%
  \def\tb@cellE{\bullet}
  \let\tb@cellD\tb@cellNF
  \def\sk{\bullet}}
\def\tb@initEmpty{%
  \def\tb@cellE{}
  \let\tb@cellD\tb@cellNF
  \def\sk{\global\let\tbcellF\tb@cellNF}}

\tb@initMedium

\def\tb@sframe#1{%
  \vbox to0pt{
    \vss
    \hbox to0pt{%
      \hss
      \vbox to\dimen1{
        \hrule depth #1 height0pt
        \vss
        \hbox to\dimen1{
          \vrule width #1 height\dimen1
          \hss
          \vrule width #1
          }%
        \vss
        \hrule height #1 depth 0in
        }%
      \kern-\tb@hframe
      }%
    \kern-\tb@hframe}}

\def\tb@hframe{.2pt}\def\tb@fframe{.4pt}\def\tb@bframe{2pt}
\def\tb@cellH{\tb@sframe{\tb@bframe}}       
\def\tb@cellNF{}                            
\def\tb@cellN{\tb@sframe{\tb@fframe}}       
\let\tbcellF\tb@cellN                       

\def\tb@Fsframe{%
  \vbox to0pt{
    \vss
    \hbox to0pt{%
      \hss
      \vbox to\dimen1{
        \fr@iftop{\hrule depth \fr@width height0pt}{\vskip \fr@width}
        \vss
        \hbox to\dimen1{
	  \fr@ifleft{\vrule width \fr@width height\dimen1}{\hskip \fr@width}
          \hss
          \fr@ifright{\vrule width \fr@width height\dimen1}{\hskip \fr@width}
          }%
        \vss
        \fr@ifbottom{\hrule height \fr@width depth 0in}{\vskip\fr@width}
        }%
      \kern-\tb@hframe
      }%
    \kern-\tb@hframe}}

\def\tb@fr{\@IFNEXTCHAR[{\tb@fra}{\global\let\tbcellF\tb@cellN}}
\def\tb@fra[#1]{%
	\global\let\fr@iftop\tb@IFNO
	\global\let\fr@ifbottom\tb@IFNO%
	\global\let\fr@ifleft\tb@IFNO%
	\global\let\fr@ifright\tb@IFNO%
	\global\let\fr@width\tb@fframe%
	\global\let\tbcellF\tb@Fsframe%
	\froptions#1\eoo
}
\def\froptions#1{\ifx#1\eoo\relax\else\froption#1\expandafter\froptions\fi}
\def\froption#1{
	\if#1t\global\let\fr@iftop\tb@IFYES\fi
	\if#1b\global\let\fr@ifbottom\tb@IFYES\fi
	\if#1l\global\let\fr@ifleft\tb@IFYES\fi
	\if#1r\global\let\fr@ifright\tb@IFYES\fi
	\if#1w\global\let\fr@width\tb@bframe\fi
}
\def\tb@IFYES#1#2{#1}
\def\tb@IFNO#1#2{#2}

\def\tb@rpad{1pt}
\def\tb@lpad{1pt}
\def\tb@tpad{1.8pt}
\def\tb@bpad{1.8pt}

\def\tb@overlay{\endcell\@IFNEXTCHAR[{\tb@overlaya}{\begincell}}
\def\tb@overlaya[#1]{\vbox to\dimen0\bgroup%
  \tb@overlayoptions#1\eoo%
  \tss\hbox to\dimen0\bgroup\lss$}
\def\tb@overlayoptions#1{\ifx#1\eoo\relax\else\tb@overlayoption#1\expandafter\tb@overlayoptions\fi}

\def\tb@overlayoption#1{
  \if#1t\def\tss{\vskip\tb@tpad}\let\bss\vss\fi
  \if#1c\let\tss\vss\let\bss\vss\fi
  \if#1b\def\bss{\vskip\tb@bpad}\let\tss\vss\fi
  \if#1l\def\lss{\hskip\tb@lpad}\let\rss\hss\fi
  \if#1m\let\lss\hss\let\rss\hss\fi
  \if#1r\def\rss{\hskip\tb@rpad}\let\lss\hss\fi
}

\def\tb@fl{\endcell\begincell\vrule depth 0pt width \dimen0 height \dimen0 \endcell\begincell}

\def\tbgobble#1{}
\def\Pscale{1}

\def\skewptn{%
  \@IFNEXTCHAR[{\tb@ptnC}{\tb@ptnC[]}}     

\def\tb@ptnC[#1](#2){%
	{%
    \let\Tscale\Pscale
    \let\\=\cr
   \def\tb@initYoung{%
	\def\tb@cell{\hskip\dimen0\tb@cellN}%
	\def\tb@kernA{\kern.5\dimen0}%
	\def\tb@kernB{\kern-.5\dimen0}%
   }%
   \def\tb@initFerrers{%
	\def\tb@cell{\hbox to\dimen0{\hss$\bullet$\hss}}%
	\def\tb@kernA{}%
	\def\tb@kernB{}%
   }%
    \dimen2=\ht\strutbox \advance\dimen2 by\dp\strutbox%
    \ifx\baselinestretch\undefined\relax%
    \else%
       \dimen0=100sp \dimen0=\baselinestretch\dimen0%
       \dimen2=100\dimen2 \divide\dimen2 by\dimen0%
    \fi%
    \let\tpos\tb@vcenter
    \tb@initYoung
    \tb@options#1\eoo
    \dimen0=\Tscale\dimen2%
    \dimen1=\dimen0 \advance\dimen1 by \tb@fframe%
    \lineskip=0pt\baselineskip=0pt
    \tpos{\skewptnDnewline#2|)}%
	}%
}%

\def\skewptnDnewline#1|{\vbox to\dimen0\bgroup\vss\tb@kernA\hbox\bgroup\skewptnEon#1,|}
\def\skewptnDendline|{\egroup\tb@kernB\vss\egroup\@IFNEXTCHAR{)}{\tbgobble}{\skewptnDnewline}}
\def\skewptnEon#1,{%
	\tb@rpN=#1%
	\ifnum#1>0
	        \loop%
		\tb@cell%
	        \ifnum\tb@rpN>1\advance\tb@rpN by-1%
        	\repeat%
	\fi%
	\@IFNEXTCHAR{|}{\skewptnDendline}{\skewptnEoff}}
\def\skewptnEoff#1,{\hskip #1\dimen0%
	\@IFNEXTCHAR{|}{\skewptnDendline}{\skewptnEon}}

\catcode`\@=\savecatcodeat
\let\savecatcodeat\undefined

\begin{document}
\newcommand{\apnd}{\mathcal{A}}
\newcommand{\rbar}{\bar{R_1}}
\newcommand{\boldx}{\textbf{x}}
\newcommand{\HRule}{\rule{\linewidth}{0.5mm}}
\newcommand{\emptbx}{\mbox{}}
\newcommand{\filbx}{\rule[-0.45mm]{4mm}{4mm}}

\begin{center}
{\Large Quantum and affine Schubert calculus and
Macdonald polynomials}

\bigskip
\bigskip

{\large Avinash J. Dalal
\footnote{Supported by the NSF grants DMS-1001898, DMS--1301695.}
 and Jennifer Morse
\footnote{Supported by the NSF grants DMS-1001898, DMS--1301695, and a Simons Fellowship.}
}

\bigskip

{\it Drexel University

Department of Mathematics

Philadelphia, PA 19104 }
\end{center}


\begin{abstract} 
We definitively establish that the theory of symmetric Macdonald polynomials 
aligns with quantum and affine Schubert calculus using 
a discovery that distinguished weak chains can be identified by
chains in the strong (Bruhat) order poset on the type-$A$ affine Weyl 
group.  We construct two one-parameter families of functions that 
respectively transition positively with Hall-Littlewood and Macdonald's 
$P$-functions, and specialize to the representatives for Schubert classes 
of homology and cohomology of the affine Grassmannian. 
Our approach leads us to conjecture that
all elements in a defining set of 3-point genus 0 
Gromov-Witten invariants
for flag manifolds can be formulated as strong covers.
\end{abstract}

\section{Introduction}

The Macdonald polynomial basis for the ring $\Lambda$ of symmetric functions 
is found at the root of many exciting projects spanning topics
such as double affine Hecke algebras, quantum relativistic 
systems, diagonal harmonics, and Hilbert schemes on points in the
plane.
The transition matrix between Macdonald and Schur functions
has been intensely studied from a combinatorial, representation theoretic, 
and algebraic geometric perspective since the time Macdonald conjectured 
\cite{[M2]} that the 
coefficients in the expansion
\begin{equation}
\label{macdointro}
H_\mu(x;q,t) = \sum_\lambda K_{\lambda\mu}(q,t) \,s_\lambda
\end{equation}
are positive sums of monomials
in $q$ and $t$ -- that is, $K_{\lambda\mu}(q,t)\in\mathbb N[q,t]$.

The Kostka-Foulkes polynomials are the special case  $K_{\lambda\mu}(0,t)$.
These appear in contexts such as Hall-Littlewood polynomials
\cite{Green}, affine Kazhdan-Lusztig theory \cite{Lu}, and affine tensor
product multiplicities~\cite{NY:1997}. Kostka-Foulkes 
polynomials also encode the dimensions of bigraded $S_n$-modules \cite{GP}.  
They were combinatorially characterized by Lascoux and Sch\"utzenberger 
\cite{LSfoulkes} who associated a non-negative integer statistic 
called $\cocharge$ to each semi-standard Young tableau and proved
that
\begin{equation}
\label{chargeintro}
K_{\lambda\mu}(0,t)=
\sum_{T\in SSYT(\lambda,\mu)}
t^{\cocharge(T)}
\,,
\end{equation}
summing over tableaux of shape $\lambda$ and weight $\mu$.
Despite the prevelance of concrete results for the $K_{\lambda\mu}(0,t)$, 
it was a big effort even to establish 
polynomiality for general $K_{\lambda\mu}(q,t)$ 
and the geometry of Hilbert 
schemes was ultimately needed to prove positivity \cite{Haiman}.
A formula in the spirit of \eqref{chargeintro} still remains a complete mystery.

One study of Macdonald polynomials \cite{LLM} uncovered
a manifestly $t$-Schur positive construction for polynomials 
$A_\mu^{(k)}(x;t)$, conjectured to be a basis
for the subspace $\Lambda_{(k)}^t$ 
in a filtration $\Lambda^{t}_{(1)}\subset\Lambda^{t}_{(2)}\subset\dots
\subset\Lambda^t_{(\infty)}$ of $\Lambda$ with the
compelling feature 
that for every partition $\mu$ where $\mu_1\leq k$,
$$
H_\mu(x;q,t)=\sum_{\lambda:\lambda_1\leq k} 
K_{\lambda\mu}^k(q,t)\,A_\lambda^{(k)}(x;t)\,
\quad\text{for some}\quad K_{\lambda\mu}^k(q,t)\in\mathbb N[q,t]\,.
$$
Assuming this conjecture,
since $A_\lambda^{(k)}(x;t)$ is a $t$-positive sum of
Schur functions,  the Macdonald/Schur transition matrices factor
over $\mathbb N[q,t]$.  The construction of $A_\mu^{(k)}(x;t)$ 
is extremely intricate and the conjectures remain unproven as a
consequence.  Nevertheless, their study inspired discoveries 
in representation theory \cite{HaChen} and suggested a connection 
between the theory of Macdonald polynomials and quantum and affine 
Schubert calculus.

The affine Grassmannian of $G=SL(n,\CC)$ is given by 
$\Gr=G(\CC((t)))/G(\CC[[t]])$, where $\CC[[t]]$ is the ring of formal 
power series and $\CC((t))=\CC[[t]][t^{-1}]$ is the ring of formal Laurent
series.  Quillen (unpublished) and Garland and Raghunathan~\cite{GRa} 
showed that $\Gr$ is homotopy-equivalent to the group 
$\Omega\, S\!U(n,\mathbb{C})$ of based loops into $S\!U(n,\mathbb{C})$.  
The homology $H_*(\Gr)$ and cohomology $H^*(\Gr)$ thus
have dual Hopf algebra structures which, using results of \cite{Bott},
can be explicitly identified by a subring $\Lambda^{(n)}$ and a quotient 
$\Lambda_{(n)}$ of $\Lambda$.

On one hand, the algebraic nil-Hecke ring construction of Kostant and 
Kumar~\cite{KK1} and the work of Peterson~\cite{Pet} developed
the study of Schubert bases associated to Schubert cells of 
$\Gr$ in the Bruhat decomposition of $G(\CC((t)))$,
\begin{equation*}
  \{\xi^w\in H^*(\Gr, \Z) \mid w\in \tS\} \quad
\text{and}\quad
  \{\xi_w\in H_*(\Gr, \Z) \mid w\in \tS\}
\,,
\end{equation*}
indexed by Grassmannian elements of the affine Weyl group $\tilde A_{n-1}$.
On the other, inspired by an empirical study of the polynomials 
$A_\lambda^{(k)}(x;t)$ when $t=1$, distinguished bases for $\Lambda^{(n)}$ 
and $\Lambda_{(n)}$ that refine the Schur basis for $\Lambda$ were 
introduced and connected to the quantum cohomology of Grassmannians 
in \cite{[LMhecke],[LMproofs]}.  The two approaches merged when
Lam proved in \cite{Lam} that these {\it $k$-Schur bases}
are sets of representatives for the Schubert classes of $H^*(\Gr)$ 
and $H_*(\Gr)$ (where $k=n-1$).  
Moreover, the $k$-Schur functions 
for $\Lambda_{(k)}=\Lambda_{(k)}^{t=1}$ were conjectured 
\cite{LMcore} to be the parameterless $\{A_\lambda^{(k)}(x;1)\}$,
suggesting a link from the theory of Macdonald polynomials to
quantum and affine Schubert calculus.

Here, we circumvent the problem that the characterization 
for $A_\mu^{(k)}(x;t)$ lacks in mechanism for proof and
definitively establish this link.
Our work relies on 
a remarkable connection between chains in the strong and the weak 
order poset on the type-$A$ affine Weyl group.
From this, we are able to construct one 
parameter families of symmetric functions that transition
positively with $H_\mu(x;0,t)$ and Macdonald's $P$-functions
and that specialize to the Schubert representatives for $H^*(\Gr)$ 
and $H_*(\Gr)$ when $t=1$.   The same approach leads also
to a strong-cover formulation for all elements in a 
defining set of Gromov-Witten invariants for flag manifolds.

Our presentation begins with a fresh look at the product structure 
of $H^*(\Gr)$ and $H_*(\Gr)$.  The ring is determined by 
multiplication of an arbitrary class with a simple class.  Explicit Pieri 
rules for these products were given in \cite{[LMproofs],[LLMS]};
the homology rule is framed using saturated chains in the weak order 
on elements in $\tilde A_{n-1}$ and the cohomology rule is in terms of 
strong order saturated chains.  We distinguish a subset of these 
strong chains by imposing a translation and a horizontality condition 
on Ferrers shapes.  We prove that this subset newly characterizes the 
homology rule, providing a cohesive framework for the 
structure of $H^*(\Gr)$ and $H_*(\Gr)$.
For $w\in \tS$ and $c_{0,m} = s_{m-1}\cdots s_1s_0$,
\begin{equation*}
\xi_{c_{0,m}}\, \xi_w = \sum_{u\in\tS} \xi_{u},
\end{equation*}
where $(\mfa(w),\mfa(u))$ is a horizontal strong $(n-1-m)$-strip.
In essence, a horizontal strong strip is a 
saturated chain in the Bruhat order from $u$ to the
translation of $w$ by $s_{x-1}\cdots s_{x+1}$
(see Definition~\ref{hstrongstrip}).

From the horizontal strong strips, we derive a new combinatorial tool 
called affine Bruhat countertableaux (or $ABC$'s).   We prove that their 
generating functions are representatives for the Schubert basis of $H^*(\Gr)$
and by associating a non-negative integer statistic called $\ncocharge$ 
to each $ABC$, we refine the Kostka-Foulkes polynomials.  
The family of {\it weak Kostka-Foulkes polynomials} are defined,
for partitions $\mu$ and $\lambda$ with parts smaller than $n$,
by
\begin{equation}
\label{kkf}
K_{\lambda\mu}^{n}(t)=\sum_{A\in ABC(\lambda,\mu)}
t^{\ncocharge(A)}\,,
\end{equation}
summing over all $ABC$'s of shape $\lambda$ and weight $\mu$.  
A new family of symmetric functions over $\mathbb Q(t)$ that reduces 
to Schubert representatives for the cohomology of $\Gr$ when $t=1$
can then be drawn by
\begin{equation}
\label{tinvdefintro}
\mathfrak S_\lambda^{(n)}(x;t) = 
\sum_{\mu} K_{\lambda\mu}^{n}(t)\, P_\mu(x;t)
\,,
\end{equation}
where $\{P_\mu(x;t)\}$ are Macdonald's $P$-functions.
A basis that reduces to the Schubert representatives
for $H_*(\Gr)$ when $t=1$ and whose transition matrix 
with Macdonald polynomials $H_\mu(x;0,t)$ has entries in 
$\mathbb N[t]$ is then given by the dual basis
$$
\left\{s_\lambda^{(n)}(x;t)\right\}\,,
$$
with respect to the Hall-inner product on $\Lambda$.
These are conjecturally the $A_\lambda^{(n-1)}(x;t)$.

Another advantage of the strongly formulated homology rule is that it allies 
with the combinatorial backdrop of quantum Schubert calculus.  
The quantum cohomology ring of the complete flag manifold
$FL_n$ (chains of vector spaces in $\mathbb C^n$)
decomposes into Schubert cells indexed by permutations $w\in S_n$.
As a linear space, the quantum cohomology is 
$QH^*(FL_n)= H^*(FL_n)\otimes\mathbb Z[q_1,\ldots,q_{n-1}]$
for parameters $q_1,\ldots, q_{n-1}$ and the appeal
lies in its rich multiplicative structure.
The quantum product
\begin{equation}
\label{qschubintro}
\sigma_u\,*\,\sigma_v = \sum_w \sum_{d} 
q_1^{d_1} \dots q_{n-1}^{d_{n-1}}\,
 \langle u,v,w\rangle_d
\, \sigma_{w_0w}
\end{equation}
is defined by 
3-point Gromov-Witten invariants of genus 0 which count
equivalence classes of certain rational curves in $FL_n$.
The study of Gromov-Witten invariants is ongoing.
Many attempts to gain direct combinatorial access to the structure 
constants have been made, but
formulas are still being pursued even in the simplest case 
when $q_1=\cdots=q_{n-1}=0$.  In this case, if $u$ and $v$ are 
permutations with one descent, the invariants reduce to the 
well-understood Littlewood-Richardson coefficients \cite{LR,SchutzLR}.  
For generic $u$ and $v$ in $S_n$, the invariants are the structure 
constants of Schubert polynomials \cite{LSschub}, mysterious
even when $u$ has only one descent.

Although the construction is not manifestly positive,
all the Gromov-Witten invariants of \eqref{qschubintro} can be 
calculated from the subset
\begin{equation}
\label{simplegwintro}
\left\{
\langle s_r,v,w \rangle_d
: 1\leq r<n\;\text{and}\;v,w\in S_n
\right\}
\,.
\end{equation}
Fomin, Gelfand, and Postnikov \cite{FGP} use quantum Schubert polynomials to
characterize this set as a generalization of Monk's formula.
%
Here, we approach the study by way of the affine Grassmannian.
Peterson asserted that $QH^*(G/P)$ of a flag variety is a quotient 
of the homology $H_*(\Gr_G)$ up to localization (detailed and proven 
in \cite{LStoda}). As a by-product,
the three-point genus zero Gromov-Witten invariants \eqref{qschubintro}
are structure constants of the Schubert basis for $H_*(\Gr)$.
A precise identification of $\langle u,v,w\rangle_d$ 
with coefficients $c_{\mu,\lambda}^\nu$ in
\begin{equation}
\xi_\mu\,\xi_\lambda  = \sum_{\nu}
c_{\mu,\lambda}^\nu\,\xi_\nu
\end{equation}
is made in \cite{LM:flag} (where $\mu,\nu,\lambda$ are
certain Ferres shapes associated to elements of $\tS$) and the
defining set \eqref{simplegwintro} of Gromov-Witten invariants 
is determined to be  a subset of
$$
\{c_{R_r',\lambda}^\nu :\;1\leq r<n\;\text{and}\;\lambda,\nu\in\Co \}\,,
$$
where $R_r'$ is the rectangular Ferrers shape $(r^{n-r})$ with its 
unique corner removed.  

In this context, the set can conjecturally be characterized
simply as strong covers under a rectangular translation;
that is,
$$
\xi_{R_r'} \,\xi_{\lambda} = \sum_{\nu\lessdot_B R(r,\lambda)}
\xi_{\nu} \,,
$$
where $\nu_i<R(r,\lambda)_i$ for some $i$ such that
$(\lambda\cup R_r)_i=r$
and the $q$-parameters are readily extracted from the shape $\nu$.
We extend the definition of horizontal strong strips to a
larger distinguished subset of strong order chains characterized
by a condition involving ribbon shapes (see Definition~\ref{rstrongstrip}).
The ribbon strong strips are inspired by the expansion of 
$\xi_{\mu}\xi_{\lambda}$, where $\mu$ is $(r^{n-r-1},r-a)$ for $1\leq a<r<n$.

Another motivation for this approach is
in its application to an open problem in the study of 
$H^*(\Gr)$.  The problem of expanding a Schubert homology class in the 
affine Grassmannian of $G=SL_{n-1}$ into Schubert homology classes in 
$\Gr$ was settled in \cite{[LLMS2]}, but the cohomological 
picture requires a deeper understanding of the intricacies of strong
strips.  The ribbon strong strips point the study towards
converting between weak and strong chains so that the existing work
on the homology problem can be applied to the cohomology problem
(see \cite{[LLMS2]} for details).

\section{Related work}

The parameterless $k$-Schur function structure constants
contain all the Schubert structure constants in the quantum 
cohomology of flag varieties \cite{LStoda,LM:flag}.  The
search for formulas for these constants is tied to many exciting projects.  

The quantum comology of the Grassmannian can be accessed \cite{BKT} from 
the ordinary cohomology of two-step flags, in which case the Schubert 
structure constants can be computed by an iterative algorithm of Coskun 
\cite{Cos} or by Knutson-Tao puzzles \cite{KnutsonTao:2003} (proved 
in forthcoming work \cite{BKPT,Buch:puz}).  The constants also match 
the structure constants of the Verlinde fusion alegbra for WZW models 
\cite{Verlinde:1988,[TUY]},
efficiently computed by the Kac-Walton formula \cite{[Kac],[Wa]} 
and combinatorially attempted by \cite{[Tu],KoSt,MS:2012} among others.

Formulas in the quantum cohomology of flag varieties have been 
derived only in special cases such as the quantum Monk formula 
\cite{FGP} and quantum Pieri formula \cite{Po:qPieri}.  These 
special constants were connected in \cite{LM:flag} to the $k$-Schur 
expansion of $s_{\mu}^{(k)}s_\lambda^{(k)}$, where $\mu$ is a rectangle 
minus part of a row.  The $k$-Pieri rule was given in \cite{[LMproofs]} 
and a more general result appears in \cite{BBPZ,BSS,BSSgen}. 
The problem currently excites many perspectives including the 
Fomin-Kirillov algebra \cite{MPP}, the affine nil-Coxeter
algebra \cite{BBTZ}, Fomin-Greene monoids \cite{BB},
residue tables \cite{FishK}, and crystal bases \cite{MS}.

The inclusion of a generic $t$-parameter has so far been met with 
limited success.  Most notably, Lapointe and Pinto \cite{LapointePinto} 
introduced a statistic on weak tableaux and proved that it matches the
weight on the poset in \cite{[LLMS2]} that describes the expansion of 
a Schubert homology class in $\Gr$ into Schubert homology classes 
in the affine Grassmannian of $G=SL(n+1,\CC)$.  In \cite{DM}, we 
prove that these match the statistic on $ABC$'s.  
Closely related is work expressing $k$-Schur functions in terms 
of Kirillov-Reshetikhin crystals for type $A_n$ \cite{MS}.

Jing \cite{Jing} introduced vertex operators $B_r$ with the property that
$$
B_r H_\mu(x;0,t)= H_{r,\mu}(x;0,t)\,.
$$
These play a central role \cite{HMZ} in the developing theory of diagonal harmonics
\cite{GHcat}.  Zabrocki \cite{Zab:thesis} determined the action of $B_r$ on 
a Schur function, giving a new proof of the $\cocharge$ formulation
for Kostka-Foulkes polynomials.  In fact, our approach to the homology Pieri 
rule using the strong instead of the weak order came out of a study of 
his work and the action of $B_r$ on a $k$-atom $A_\lambda^{(k)}(x;t)$.
A deeper understanding of the operators will shed light on 
open problems in diagonal harmonics and their connection to
affine Schubert calculus.

\section{Preliminaries}

The type-$A$ affine Weyl group is realized as the 
{\it affine symmetric group} $\tilde S_{n}$ given by 
generators $\{s_0,s_1,\ldots,s_{n-1}\}$ satisfying
the relations
\begin{eqnarray*}
& s_i^2 = 1,\nonumber\\
& s_i s_{i+1} s_i = s_{i+1} s_i s_{i+1}, \label{eq:affsymgens}\\
& s_i s_j = s_j s_i\quad \hbox{ for }i-j \not\equiv 1,n-1 \pmod{n}\nonumber
\end{eqnarray*}
with all indices related $\text{mod }n$.  If 
$w=s_{i_1}\cdots s_{i_\ell}\in \tilde S_n$ and
$\ell$ is minimal among all such expressions for $w$,
then $s_{i_1}\cdots s_{i_\ell}$ is called a reduced word
for $w$ and the length of $w$ is defined by $\ell(w)=\ell$.
The weak order on $\tilde S_n$ is defined by 
the covering relation
$$
w\lessdot z\iff
z=s_iw \;\text{and}\; \ell(z)=\ell(w)+1\,.
$$

Alternatively, $\tilde S_n$ is the group of permutations
of ${\mathbb Z}$ with the property that $w\in {\tilde S}_n$ acts by
$w( i + rn ) = w(i) + rn$, for all $r \in \ZZ$ and
$\sum_{i=1}^n w(i)=\binom{n+1}{2}$.
For $0\le i< n$, the elements $s_i \in {\tilde S}_{n}$ 
act on $\ZZ$ by $s_i( i+rn ) = i+1+rn$, $s_i( i+1+rn ) = i+rn$, and $s_i( j ) = j$ 
for
$j \not\equiv i,i+1 \pmod{n}$.  

Although the simple reflections
$s_i$ generate the group, there is also the notion of a transposition 
$\tau_{i,j}$ defined by its action $\tau_{i,j}( i +rn ) = j + rn$
and $\tau_{i,j}( \ell ) = \ell$ for $\ell \not\equiv i,j~\pmod{n}$.
Take integers $i < j$ with $i \not\equiv j \pmod{n}$ 
and $v = \lfloor (j-i)/n \rfloor$, then $\tau_{i,i+1} = s_i$ and for
$j-i>1$,
\[
\tau_{i,j} = s_i s_{i+1} s_{i+2} \cdots s_{j-v-2} s_{j-v-1} s_{j-v-2} s_{j-v-3} \cdots s_{i+1} s_i
\]
where the indices of simple reflections are taken $\text{mod } n$.  
For $j>i$, we set $\tau_{i,j} = \tau_{j,i}$. 
The strong (Bruhat) order is defined by the covering relation
\begin{equation}
w \lessdot_B u \quad\text{when} \quad u=\tau_{r,s} w
\quad\text{and} \quad \ell(u)=\ell(w)+1\,.
\end{equation}

The symmetric group $S_{n}$ can be viewed as the parabolic subgroup
of $\tilde S_n$ generated by $\{ s_1, s_2, \ldots, s_{n-1} \}$.
The left cosets of ${\tilde S}_{n} / S_{n}$ are called
\textit{affine Grassmannian elements} and they are 
identified by the set $\tilde S_n^0\subset \tilde S_n$
of minimal length coset representatives.
Elements of $\tS$ can be conveniently represented by 
a subset of partitions.
A partition is a non-decreasing vector $\lambda=(\lambda_1,\ldots, \lambda_n)$ 
of positive integers and it is identified by its (Ferrers) shape having 
$\lambda_i$ lattice squares in the $i^{th}$ row, 
from the bottom to top.  For partitions $\lambda$ and $\mu$, 
$\mu \subset \lambda$ when $\mu_i \leq \lambda_i$.  
Given $\mu\subset\lambda$, the skew shape $\lambda/\mu$ is 
defined by deleting the cells of $\mu$ from $\lambda$.
When there is at most one cell in any column of 
$\lambda/\mu$, this skew shape with $m$ cells is called
a \textit{horizontal $m$-strip}.
The cell of a horizontal 1-strip $\lambda/\mu$ is called a
\textit{removable corner} of $\lambda$ and an \textit{addable
corner} of $\mu$.  A cell $(i,j)$ of a partition
$\lambda$ with $(i+1,j+1) \not \in \lambda$ is called an \textit{extremal cell}.

It is the subset of shapes $\mathcal C^n$ called $n$-cores that are 
in bijection with affine Grassmannian permutations.
An \textit{$n$-core} is a partition 
that has no cell whose hook-length is $n$ (the hook-length
of cell $c=(i,j)$ is $\lambda_i-i+1$ plus
the number of cells in colum $j$ above cell $c$).
The \textit{content} of $c$ is $j-i$ and its \textit{$n$-residue}
is $j-i \pmod n$.  
There is a left action of the generator $s_i\in \tilde S_{n}$
on an $n$-core $\lambda$ with at least one addable corner of 
residue $i$; it is defined by letting $s_i \lambda$ be the shape 
where all corners of residue $i$ have been added to $\lambda$.
This extends to a bijection \cite{LasOrder,LMcore}
$$
\mfa: \tS \longrightarrow \mathcal C^n \,,
$$
where $\lambda=\mfa(w)=s_{i_1}\cdots s_{i_\ell}\emptyset$
for any reduced word ${i_1}\cdots {i_\ell}$ of $w$.
We use $w_\lambda$ to denote the preimage of $\lambda$
under $\mfa$.  From this, we define the {\it degree} of the
$n$-core $\lambda$, $deg(\lambda)$,  to be $\ell=\ell(w_\lambda)$.
An $n$-core $\lambda$ has an addable corner of residue $i$
if and only if 
\begin{equation}
\label{addres}
\ell(w_{s_i\lambda})=
\ell(w_\lambda)+1 \,.
\end{equation}

\begin{property}
\cite{LMcore}
\label{thecoreprop}
For an $n$-core $\lambda$ with extremal cells 
$c$ and $c'$ of the same $n$-residue, given that
$c$ is weakly north-west of $c'$,
if $c$ is at the end of its row, then so is $c'$.
If $c$ has a cell above it, then so does $c'$.
\end{property}

The strong order on the subset $\tS$ is characterized on
elements of $\mathcal C^n$ by the containment of
diagrams and its covering relation is
\begin{displaymath}
  \mu \lessdot_B \lambda \Longleftrightarrow \mu \subset \lambda
\text{ and } deg(\lambda) = deg(\mu)+1\,.
\end{displaymath}
Given a pair $\mu~\lessdot_B~\lambda$, the shape $\lambda/\mu$
can be described by ribbons.
For $r \geq 0$, an \textit{$r$-ribbon} $R$ is a 
skew diagram $\lambda/\mu$ consisting of $r$ rookwise connected 
cells such that there is no $2 \times 2$ shape contained in $R$.  
In a ribbon, the southeasternmost cell is called its \textit{head}
and the northweasternmost cell is its \textit{tail}.
The height of a ribbon is the number of rows it occupies.
\begin{lemma}\cite{[LLMS]}
\label{lem:rib}
If $w\lessdot_B \tau_{r,s}w$ in $\tS$, then the skew 
$\mfa(\tau_{r,s}w)/\mfa(w)$
is made up of copies of one fixed ribbon such that the head of each
copy has residue $s-1$ and the tail has residue $r$.
\end{lemma}

\begin{lemma}
\label{lem:tau}
Given $w\lessdot_B \tau_{r,s}w$ in $\tS$, 
$$
\mfa(\tau_{r,s} w) = \mfa(w) + \,\text{all addable ribbons
with a head of residue $r-1$ and tail of residue $s$}
\,.
$$
\end{lemma}
\begin{proof}
Since $\tau_{r,s} = s_rs_{r+1}\cdots s_{s-1}\cdots s_r$,
any addable ribbon of $\mfa(w)$ with a head of residue $s-1$ and
a tail of residue $r$ is added to $\mfa(w)$ under multiplication
by $\tau_{r,s}$.  The result then follows from Lemma~\ref{lem:rib}.
\end{proof}

%


\section{Affine Pieri rules}

The main discovery is that there is a fundamental connection between 
weak order chains from $u$ to $v$ in $\tS$ and strong order 
chains from $v$ to a translation of $u$.  We start by 
addressing the case that applies to Pieri rules, in which case the
translation of an element $u_\lambda\in\tS$ amounts to 
$s_{x-1}s_{x-2}\cdots s_{x+1}u_\lambda$
where $x=\lambda_1-1\pmod n$.

\subsection{Strong and weak Pieri rules}

The Pieri rules for the $H^*(\Gr)$ and $H_*(\Gr)$ are given in
\cite{[LMproofs],[LLMS]}.  
The affine homology rule is framed
using saturated chains in the weak order on $\tS$, whereas
the cohomology rule is in terms of strong order saturated chains.

A word $a_1a_2 \cdots a_{\ell}$ with letters in $\mathbb Z/n\mathbb Z$
is called {\it cyclically decreasing} if each letter occurs at most
once and $i+1$ precedes $i$ whenever $i$ and $i+1$ both occur in the word.
An affine permutation is called cyclically decreasing
if it has a cyclically decreasing reduced word.
The affine homology Pieri rule for $H_*(\Gr)$ is given,
for $w\in\tS$ and $c_{0,m} = s_{m-1}\cdots s_1s_0 \in \tS$,
by
\begin{equation}
\label{weakpieri}
\xi_{c_{0,m}}\, \xi_w = \sum_{v} \xi_{vw},
\end{equation}
over all cyclically decreasing $v$ of length $m$ such that $vw\in\tS$
and $\ell(vw)=\ell(w)+m$.  Thus, for $u\in\tS$,
the term $\xi_u$ occurs in the product of $\xi_{c_{0,m}}\xi_w$ 
only when $uw^{-1}$ is cyclically decreasing and 
there is a saturated chain
$$
w=w^{(0)}\lessdot w^{(1)}\lessdot\cdots\lessdot w^{(m)}=u\,.
$$
Alternatively, the rule can be formulated in the language of shapes using
the action of the $s_i$-generators on $n$-cores.
\begin{lemma}\cite{[LLMS]}
\label{lem:cychor}
For $u,w \in \tS$ where $\ell(uw^{-1})=m$, $uw^{-1}$ is cyclically 
decreasing with reduced word $j_1\cdots j_m$ if and only if
$\mfa(u)/\mfa(w)$ is a horizontal strip such that the set of residues
labelling its cells is $\{j_1,\ldots,j_m\}$.
\end{lemma}

While the weak order determines the affine homology rule
for $H_*(\Gr)$, the affine cohomology Pieri rule 
is given as the sum over certain multisets of chains in the 
Bruhat (strong) order.  The multisets arise by imposing a marking on 
strong covers $\rho\lessdot_B\gamma$ in $\mathcal C^{n}$.  
Define $(\gamma,c)$ to be a {\it marked strong cover} of $\rho$
if $\rho \lessdot_B \gamma$ and $c$ is the content
of the head of a ribbon in $\gamma/\rho$ 
(recall that Lemma~\ref{lem:rib} assures the skew shape
is made up of ribbons).
Then, for $0 \leq m < n$ and $n$-cores $\nu$ and $\gamma$,
a {\it strong $m$-strip} from $\nu$ to $\gamma$ is a saturated chain of cores
\begin{displaymath}
\nu= \gamma^{(0)} \lessdot_B \gamma^{(1)} \lessdot_B \ldots
\lessdot_B \gamma^{(m)} = \gamma\,,
\end{displaymath}
together with an increasing ``content vector" $c =
(c_1,c_2,\cdots,c_{m})$, such that
$(\gamma^{(i)},c_i)$ is a marked strong cover of $\gamma^{(i-1)}$ for $1
\leq i \leq m$.
\begin{example}
\label{ex:strong}
For $n=4$, there are 2 saturated chains from $\nu = (3)$ to
$\gamma=(4,1,1)$,
\begin{equation}
\label{ex:strongchain}
\text{\tiny\tableau*[sbY]{ & & }} \lessdot_B
\text{\tiny\tableau*[sbY]{ \fr[w,l,t,r] \cr \fr[w,l,b,r] \cr & & }}
\lessdot_B \text{\tiny\tableau*[sbY]{ \cr \cr & & & \tf}}
\qquad
\text{\tiny\tableau*[sbY]{ & & }} \lessdot_B
\text{\tiny\tableau*[sbY]{ \tf \cr & & & \tf}} \lessdot_B
\text{\tiny\tableau*[sbY]{ \tf \cr \cr & & & }}
\,.
\end{equation}
The first chain with content vector $c=(-1,3)$ is thus the only
strong 2-strip from $\nu$ to $\gamma$.
\end{example}
The affine cohomology Pieri rule is
\begin{equation}
\label{strongpieri}
\xi^{c_{0,m}}\, \xi^w = \sum_{S} \xi^{z}\,,
\end{equation}
where the sum runs over strong $m$-strips $S$ from
$\mfa(w)$ to $\mfa(z)$.  In contrast to the Pieri rule
for $H_*(\Gr)$, a given term $\xi^z$ here
may occur with multiplicity greater than $1$.

\subsection{Horizontal strong strips}
\label{sec:pieri}

As with the affine Pieri rule for the cohomology $H^*(\Gr)$,
the Pieri rule~\cite{BS} and the quantum Pieri rule~\cite{FGP}
for (quantum) cohomology of the flag manifold are also determined 
by chains in the Bruhat order (see \eqref{quantummonk}).
However, it is the homology of $\Gr$, not the cohomology,
that is algebraically tied to the quantum cohomology  of the flag manifold
(detailed in \S~\ref{sec:gw}).  To align the combinatorics with the algebra,
we introduce a distinguished subclass of strong order chains that
characterize the affine homology Pieri rule.
The fundamental observation is that the translation of an $n$-core  
$\lambda$ to the $n$-core $R(n-1,\lambda)=(\lambda_1+n-1,\lambda)$ 
plays a crucial role.

\begin{definition}\label{hstrongstrip}
A pair of $n$-cores $(\lambda,\nu)$
is a horizontal strong $m$-strip if $\lambda\subset\nu$
and there is a saturated chain of cores
\begin{equation}
\label{schain}
\nu=\nu^{(0)}\lessdot_B\nu^{(1)}
\lessdot_B \cdots \lessdot_B \nu^{(m)}=R(n-1,\lambda)
\end{equation}
such that the bottom row of $\nu^{(i)}$ is longer than the bottom
row of $\nu^{(i-1)}$, for $1 \leq i \leq m$ where $m=n-1+deg(\lambda)-deg(\nu)$.
\end{definition}

\begin{example}
\label{ex:stronghs}
For $n=4$, $\lambda = (1,1),$ and $\nu = (3)$,
$(\lambda,\nu)$ is not a horizontal strong strip
since neither of the strong chains from $\nu$ to $R(n-1,\lambda)$
shown in \eqref{ex:strongchain} have strictly growing bottom 
rows.  

\noindent
For $\lambda = (3,1,1)$, the $4$-cores $\nu$ 
such that $(\lambda,\nu)$ is a horizontal strong $2$-strip
begin with the chains to $R(3,\lambda)$:
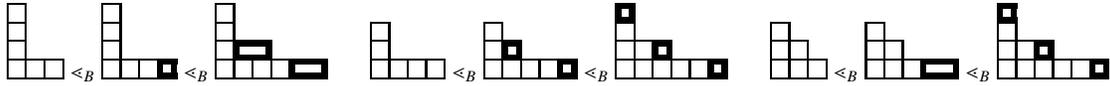
$$
{\substack{{\text{\tiny\tableau*[sbY]{ \cr \cr \cr & & }}} ~\lessdot_B~
{\text{\tiny\tableau*[sbY]{\cr \cr \cr & & & \tf}}} ~\lessdot_B~
{\text{\tiny\tableau*[sbY]{ \cr \cr & \fr[w,l,t,b] & \fr[w,t,b,r] \cr & & & &
\fr[w,l,t,b] & \fr[w,t,b,r]}}} 
}}
\hspace{0.23in}
{\substack{{\text{\tiny\tableau*[sbY]{ \cr \cr & & & }}} ~\lessdot_B~
{\text{\tiny\tableau*[sbY]{\cr & \tf \cr & & & & \tf}}} ~\lessdot_B~
{\text{\tiny\tableau*[sbY]{ \tf \cr \cr & & \tf \cr & & & & & \tf}}}
}}
\hspace{0.23in}
{\substack{{\text{\tiny\tableau*[sbY]{ \cr & \cr & &}}} ~\lessdot_B~
{\text{\tiny\tableau*[sbY]{\cr & \cr & & & \fr[w,l,t,b] & \fr[w,t,b,r]}}}
~\lessdot_B~ {\text{\tiny\tableau*[sbY]{ \tf \cr \cr & & \tf \cr & & & & & \tf}}}
}}
$$
\end{example}

We have chosen the terminology horizontal strong strip
because, although not immediately obvious, there always exists a strong strip 
from $\nu$ to $R(n-1,\lambda)$ of shapes that differ 
by ribbons of height one when $(\lambda,\nu)$ is a horizontal strong strip. 
The following lemma associates horizontal strong strips to
the horizontality condition and we then connect to strong strips.

\begin{lemma}
\label{lem:hstrip}
Given $n$-cores $\lambda\subset\nu$ and a saturated chain
of shapes \eqref{schain} whose bottom rows strictly 
increase, there are adjacent ribbons $S^1,\ldots,S^m$ in the 
bottom row of $R(n-1,\lambda)/\lambda$ such that
the shape $\nu^{(j)}/\nu^{(j-1)}$ is comprised 
of all copies of $S^j$ that can be removed from $\nu^{(j)}$,
for each $1\leq j\leq m$.
\end{lemma}
\begin{proof}
Consider $\lambda\subset\nu$ and a chain  of $n$-cores
\eqref{schain} where the bottom rows increase in size.
Let $S^j$ denote the lowest ribbon in $\nu^{(j)}/\nu^{(j-1)}$.
Since the bottom row of $\nu^{(j)}$ is strictly longer
than the bottom of $\nu^{(j-1)}$, the head of $S^j$
lies in the bottom row of $\nu^{(j)}$.  Moreover, $S^j$ has height 
one since $\lambda\subset\nu^{(j)}\subset
(n-1+\lambda_1,\lambda)$ and $(n-1+\lambda_1,\lambda)/\lambda$ 
is a horizontal strip.
Therefore, $S^j$ is a removable ribbon lying entirely in the bottom 
row of $\nu^{(j)}$.
Lemma~\ref{lem:tau} then implies that
$\nu^{(j)}/\nu^{(j-1)}$ consists of all copies of $S^j$
that can be removed from $\nu^{(j)}$.
\end{proof}

\begin{proposition}
\label{prop:ss2rss}
For $n$-cores $\lambda\subset\nu$,
the pair $(\lambda,\nu)$ is a horizontal strong $m$-strip
if and only if there exists a strong
$m$-strip from $\nu$ to $R(n-1,\lambda)$
whose content vector $c$ satisfies $c_1\geq \lambda_1$.
\end{proposition}

\begin{proof}
Given any horizontal strong $m$-strip $(\lambda,\nu)$,
we have a chain \eqref{schain} that is characterized
by ribbons $S^1,\ldots,S^m$ lying in the bottom row 
of $R(n-1,\lambda)/\nu$ by Lemma~\ref{lem:hstrip}.
We can obtain a strong strip by associating it to 
the content vector $(c_1,\ldots,c_m)$, where $c_i$ is 
the content of the head of ribbon $S^i$.
Then $c_1 \geq \lambda_1$ since $\lambda\subset\nu$.

On the other hand, consider cores $\nu=\nu^{(0)}\lessdot_B\nu^{(1)}
\lessdot_B \cdots \lessdot_B \nu^{(m)}=R(n-1,\lambda)$
such that the head $h_i$ of a ribbon in $\nu^{(i)}/\nu^{(i-1)}$
has content $c_i$ and $\lambda_1\leq c_1<\cdots <c_m$.
The last $n-1$ cells in the bottom row of
$R(n-1,\lambda)$ lie at the top of their column and 
therefore they are the only cells with content greater than $\lambda_1-1$.
Therefore, $h_i$ must lie in the bottom row of
$\nu^{(i)}\subset R(n-1,\lambda)$ implying that
bottom rows are strictly growing.
\end{proof}

\begin{remark}
\label{rem:unique}
The proof of Theorem~\ref{prop:main} will establish
a claim stronger than Proposition~\ref{prop:ss2rss}:
each horizontal strong strip $(\lambda,\nu)$ corresponds uniquely 
to a strong strip from $\nu$ to $R(n-1,\lambda)$ with $c_1\geq \lambda_1$.
\end{remark}

\begin{example}\label{stronghsEx}
For $n=4$ and $\lambda = (3,1,1)$, the $n$-cores $\nu$ 
such that $(\lambda,\nu)$ is a horizontal strong $2$-strip
are given in Example~\ref{ex:stronghs} and each corresponds to 
a unique strong 2-strip from $\nu$ to $R(3,\lambda)$ 
with $c_1\geq 3$: their content vectors are 
$(3,5)$, $(4,5)$, and $(4,5)$, respectively.
\end{example}

Horizontal strong strips in hand, we now discuss their
correspondence with weak order cyclically decreasing chains.
For a fixed $x\in [n]=\{0,\ldots,n-1\}$ and $y \in \{0,\ldots,n-1\}\backslash\{x\}$, 
it will be inferred that $x+1 \leq y \leq x-1$ is taken with
respect to the total order defined by
$$
x+1 < x+2 < \cdots < 0 < n-1 < \cdots < x-1.
$$
For $n$-cores $\lambda$ and $\nu$,
a simple construction produces a cyclically decreasing 
word for $w_\nu w_\lambda^{-1}$ from a relevant strong chain 
from $\nu$ to $R(n-1,\lambda)$.  For $x=\lambda_1-1\pmod n$,
define the map
$$
\psi:
\nu=\nu^{(0)}\lessdot_B\nu^{(1)}\lessdot_B\cdots\lessdot_B\nu^{(m)}=R(n-1,\lambda)
\quad
\longmapsto \quad s_{j_1}\cdots s_{j_{n-1-m}}
$$
where the elements $j_1>\cdots >j_{n-1-m}$ of
$\{x-1,\ldots,x+1\}\backslash\{a_1,\ldots,a_m\}$ are
obtained by taking $a_{m-i}$ to be the residue of 
the leftmost cell in the
bottom row of $\nu^{(i+1)}/\nu^{(i)}$, for $0\leq i<m$.
In reverse, a strong chain arises from a reduced 
expression for $w_\nu w_{\lambda}^{-1}$ with the map
$$
\phi:
s_{j_1}\cdots s_{j_{n-1-m}}\quad\longmapsto\quad
\nu=\nu^{(0)}\subset\nu^{(1)}\subset\cdots\subset\nu^{(m)} =R(n-1,\lambda)
\,,
$$
where $\nu^{(i)}$ is obtained from $\nu^{(i+1)}$ by deleting 
all removable copies of the ribbon whose tail has residue $a_{m-i}$
and lies in the bottom row,
where $x+1 \leq a_m < \cdots < a_1 \leq x-1$ are the elements
of $\{x-1,\ldots,x+1\}\backslash\{j_1,\ldots,j_{n-1-m}\}$.

Several lemmas are first needed to prove that $\phi$ and $\psi$ 
give the desired bijection.
Horizontal strong strips $(\lambda,\nu)$ are defined on the level 
of cores where the key idea is to study strong chains from 
$\nu$ to the $n$-translation of $\lambda$ defined by $R(n-1,\lambda)$.
A preliminary result puts the idea of this translation into the 
framework of the affine Weyl group.

\begin{lemma}
\label{lem:wRkw}
For $w_\lambda\in \tS$, the length $\ell(w_{R(n-1,\lambda)})=n-1+\ell(w_\lambda)$
and
$$
w_{R(n-1,\lambda)} = s_{x-1}\cdots s_{x+1}\,w_\lambda\,,
$$
where $x=\lambda_1-1\pmod n$.
\end{lemma}
\begin{proof}
It suffices to prove that $R(n-1,\lambda)=\mfa(s_{x-1}\cdots s_{x+1}w)$.
Since the lowest addable corner of $\lambda$ has residue $x+1$,
$s_{x+1}$ acts on $\lambda$ by adding all corners of residue $x+1$.
Similarly, $s_{x+2}$ adds corners of
residue $x+2$ and by iteration, the degree of $\lambda$
increases by $n-1$ under the action of $s_{x-1}\cdots s_{x+1}$.
Since $s_{x-1}\cdots s_{x+1}$ is cyclically decreasing,
Lemma~\ref{lem:cychor} implies that it
acts on $\lambda$ by adding a horizontal strip.
The result follows by noting that
$R(n-1,\lambda)$ is the unique core obtained
by adding a horizontal strip to $\lambda$ and increasing degree by $n-1$.
\end{proof}

For $x \in \{0,1,\ldots,n-1\}$, let
$S_{\hat{x}} = \langle s_0,\ldots,\hat s_{x},\ldots, s_{n-1}\rangle
\subset\tS$ be the subgroup generated by all simple reflections
except $s_x$.

\begin{lemma}
\label{lem:tildev}
Given $w_\lambda,u\in\tS\,$ where
$u w_\lambda^{-1}$ is a cyclically decreasing permutation
and $\ell(uw_\lambda^{-1})=\ell(u)-\ell(w_\lambda)$,
then $uw_\lambda^{-1}\in S_{\hat x}\,$ for $x=\lambda_1-1\pmod n$.
\end{lemma}
\begin{proof}
Let $v=s_{j_1}\cdots s_{j_m}$ be a reduced expression for $uw_\lambda^{-1}$.
By the definition of $\mfa$, the residues labelling the cells 
in $D=\mfa(vw_\lambda)/\lambda$ come from the set $\{j_1,\ldots,j_m\}$.
In fact, since $\ell(vw_\lambda)=\ell(w_\lambda)+m$,
the cells of $D$ are labelled by precisely the set $\{j_1,\ldots,j_m\}$.
Since $v$ is cyclically decreasing, we also have that $D$ is a horizontal
strip by Lemma~\ref{lem:cychor}.  Therefore, an extremal
cell of residue $j_t$ that does not lie at the end of its row
occurs in $\lambda$ for every $1\leq t\leq m$.
Since $x$ is the residue of the last cell in the bottom row
of $\lambda$, Property~\ref{thecoreprop} implies that
every extremal cell of $\lambda$ with residue $x$ lies at
the end of its row.  In particular, $x\neq j_t$ and
we have $v\in S_{\hat x}$.
\end{proof}

\begin{theorem}
\label{prop:main}
For $n$-cores $\lambda$ and $\nu$,
$(\lambda,\nu)$ is a horizontal strong strip if and only if
$w_\nu w_{\lambda}^{-1}$ is a cyclically decreasing permutation 
where $\ell(w_\nu)=\ell(w_\lambda)+\ell(w_\nu w_\lambda^{-1})$.
\end{theorem}
\begin{proof}
From a horizontal strong $m$-strip $(\lambda,\nu)$, Lemma~\ref{lem:hstrip} 
gaurantees us a chain
$\nu=\nu^{(0)}\lessdot_B\nu^{(1)}\lessdot_B\cdots\lessdot_B\nu^{(m)}=R(n-1,\lambda)
\,,$
such that a ribbon $S^{i+1}$ of $\nu^{(i+1)}/\nu^{(i)}$ has
height one and is a removable ribbon in the bottom row of $\nu^{(i+1)}$.
It suffices to prove that the image $s_{j_1}\cdots s_{j_{t}}$
of this chain under $\psi$ is a
cyclically decreasing word for $w_\nu w_{\lambda}^{-1}$
of length $n-1-m$
since the definition of horizontal $m$-strip implies that
$\ell(w_\nu)=n-1-m+\ell(w_\lambda)$.

The definition of $\psi$ uses $a_{m-i}$ 
to denote the residue of the tail of $S^{i+1}$ and thus 
the residue of the head of $S^{i}$ must be $a_{m-i}-1$.
By Lemma~\ref{lem:rib}, we have
$w_{\nu^{(i)}}=\tau_{a_{m-i},a_{m-i-1}}w_{\nu^{(i+1)}}$ 
for $0 \leq i < m$, where $a_0 = \lambda_1 - 1 \pmod n$.
In particular, $w_\nu=\tau_{a_m,a_{m-1}}
\cdots \tau_{a_1,a_0} w_{R(n-1,\lambda)}$.
Since $\lambda \subset \nu$, we have that $\lambda_1\leq a_m$ and 
therefore $x+1 \leq a_m < \cdots < a_1 \leq x-1$ for $x=\lambda_1-1\pmod n$.
It follows from Lemma~\ref{lem:wRkw} that
$$
w_\nu w_\lambda^{-1}=\tau_{a_m,a_{m-1}}
\cdots \tau_{a_1,x}\left(s_{x-1}\cdots s_{x+1}\right)\,,
$$
or $w_\nu w_\lambda^{-1}=s_{j_1}\cdots s_{j_{n-1-m}}$
where $j_1>\cdots >j_{n-1-m}$ are the 
elements of $\{x-1,\ldots,x+1\}\backslash\{a_1,\ldots,a_m\}$.
Since these are $n-1-m$ distinct elements, the expression is reduced.

Before proving the reverse direction, note that 
$s_{j_1}\cdots s_{j_{n-1-m}}$ is the unique reduced expression for 
$w_\nu w_\lambda^{-1}$ that is ordered by 
$x-1\geq j_1>\cdots >j_{n-1-m}\geq x+1$ and it is determined
uniquely from ribbon tails in the given chain.  
Since a given chain under consideration is determined uniquely 
by its ribbon tails, the uniqueness claim of Remark~\ref{rem:unique} 
follows.

Suppose now that $j_1\cdots j_{n-1-m}$ is a reduced word for
a cyclically decreasing permutation $w_\nu w_{\lambda}^{-1}$
where $\ell(w_\nu)=\ell(w_\lambda)+n-1-m$.  By Lemma~\ref{lem:tildev},
$w_\nu w_\lambda^{-1}\in S_{\hat x}$ for $x=\lambda_1-1\pmod n$
and therefore there are $m$ elements $x-1\geq a_1>a_2>\cdots >a_m\geq x+1$
in the set $\{x-1,\ldots,x+1\}/\{j_1,\ldots,j_{n-1-m}\}$.
The $n-1$ removable cells in the bottom row of $R(n-1,\lambda)$,
of residues $x-1,\ldots,x+1$ from right to left,
can thus be tiled uniquely into ribbons whose tails are
$a_1,\ldots,a_m$, from right to left.  Therefore,
the shapes in the image of $j_1\cdots j_{n-1-m}$
under $\phi$
$$
\nu=\nu^{(0)}\subset\nu^{(1)}\subset\cdots\subset\nu^{(m)}
=R(n-1,\lambda)
$$
have increasing bottom rows. We claim
this is a strong saturated chain and $\lambda\subset\nu$.

Let $\eta^{(m)}=\nu^{(m)}$ so that by
Lemma~\ref{lem:wRkw}, $w_{\eta^{(m)}}=(s_{x-1}\cdots s_{x+1})w_\lambda$.
For $1\leq i\leq m$, define
$$
w_{\eta^{(m-i)}} = \tau_{a_i,a_{i-1}}
w_{\eta^{(m-i+1)}} =
(s_{x-1}\cdots \hat{s}_{a_1} \cdots \hat{s}_{a_2} 
\cdots \cdots  \hat{s}_{a_i} \cdots s_{x+1})w_{\lambda}\,,
$$
where $a_0=x$. 
Since $w_{\eta^{(0)}}=s_{j_1}\cdots s_{j_{n-1-m}}w_\lambda$,
we have that
$\lambda\subset \eta^{(0)}$ by Lemma~\ref{lem:cychor}.
If $w_{\eta^{(m-i)}}\lessdot_B w_{\eta^{(m-i+1)}}$,
then $\eta^{(m-i)}=\nu^{(m-i)}$ by Lemma~\ref{lem:tau}
and the claim follows.  To ensure that 
$w_{\eta^{(m-i)}}\lessdot_B w_{\eta^{(m-i+1)}}$,
it suffices to show that $w_{\eta^{(m-i+1)}}$
has length $n-i+\ell(w_\lambda)$.
Note that $\ell(w_{\eta^{(0)}})=n-1-m+\ell(w_\lambda)$
and consider $w_{\eta^{(m-i)}}$
of length $n-1-i+\ell(w_\lambda)$.  
By commuting relations,
$w_{\eta^{(m-i)}} = 
(\hat{s}_{a_{i-1}} \cdots  \hat{s}_{a_{i}}) w_{\mu}$
for
$w_\mu= (s_{x-1}\cdots \hat{s}_{a_1} \cdots \cdots \hat{s}_{a_{i-1}})
(\hat s_{a_i}\cdots s_{x+1})w_\lambda$.
Since the lowest addable corner of $\lambda$ has residue $x+1$,
the lowest addable corner of $\mu$ has residue $a_i$.
Therefore,
$w_{\eta^{(m-i+1)}} = 
(\hat{s}_{a_{i-1}} \cdots s_{a_i+1}  s_{a_i}) w_{\mu}$
has length $n-i+\ell(w_\lambda)$.  
\end{proof}

\begin{corollary}
\label{thm:Pieristrong}
For $1\leq m< n$ and $w\in \tS$,
\begin{equation}
\label{Pieristrong}
\xi_{c_{0,m}}\, \xi_w = \sum_{u\in \tS} \xi_{u}\,,
\end{equation}
where the sum is over $u$ such that $(\mfa(w),\mfa(u))$
is a horizontal strong $(n-1-m)$-strip.
\end{corollary}
\begin{proof}
A term $v=w_\nu w_\lambda^{-1}$
occurs in the summand of \eqref{weakpieri}
if and only if it is cyclically decreasing
of length $m$ and $\ell(w_\nu)=\ell(w_\lambda)+m$.
That is, if and only if $(\lambda,\nu)$ is
a horizontal strong $n-1-m$-strip by
Theorem~\ref{prop:main}.
\end{proof}

\begin{example}\label{strongprod}
The expansion $\xi_{c_{0,2}}\xi_{(3,1,1)}=
\xi_{(3,1,1,1)}+\xi_{(4,1,1)}+\xi_{(3,2,1)}$
follows from Example~\ref{ex:stronghs}
and Corollary~\ref{thm:Pieristrong}.
\end{example}

\section{Explicit representatives for Schubert classes}

The perspective of horizontal strong strips applies to the study of 
the (co)homology classes of the affine Grassmannian.
Here we derive a new combinatorial object with which to study the 
representatives for Schubert classes of $H^*(\Gr)$ and $H_*(\Gr)$.

\subsection{Polynomial realization of $H^*(\Gr)$ and $H_*(\Gr)$}

Quillen (unpublished) and Garland and Raghunathan~\cite{GRa} showed that $\Gr$ 
is homotopy-equivalent to the group $\Omega\, S\!U(n,\mathbb{C})$ of based
loops into $S\!U(n,\mathbb{C})$.  Results from \cite{Bott} can be
used to obtain a polynomial identification of $H^*(\Gr)$ and $H_*(\Gr)$ 
inside the ring of symmetric functions 
$\Lambda=\mathbb Z[h_1,h_2,\ldots,]$, where
$h_r=\sum_{1\leq i_1\leq \cdots\leq i_r} x_{i_1}\cdots x_{i_r}$.

Traditionally, bases for the space of symmetric function are
indexed by partitions.  Descriptions of the homology and 
cohomology ring are most natural in terms of 
the functions defined by setting $h_\lambda=h_{\lambda_1}\cdots h_{\lambda_\ell}$
and the monomial 
symmetric functions $m_\lambda$, defined for each 
partition $\lambda$ as the sum over $x^\alpha$ for
each distinct rearrangement $\alpha$ of the parts of $\lambda$.
The homology $H_*(\Gr)$ 
is identified by the subring $\Lambda_{(n)}$ of $\Lambda$
and the cohomology $H^*(\Gr)$ can be identified by the 
quotient $\Lambda^{(n)}$ where
$$
\Lambda_{(n)} = \mathbb Z[h_1,\ldots,h_{n-1}]
\quad \text{and}\quad
\Lambda^{(n)}=
\Lambda/\langle m_\lambda : \lambda_1\geq n\rangle\,.
$$
These spaces are naturally paired under the
Hall-inner product on $\Lambda$, defined by setting
$$
\langle h_\lambda,m_\mu\rangle =
\delta_{\lambda\mu}
$$
where $\delta_{\lambda\mu}=0$ when $\lambda\neq \mu$ and is 1 otherwise.

The Schur function basis for $\Lambda$ is self-dual with respect to $\langle,\rangle$.
Recall that this basis is a fundamental combinatorial tool to study tensor products 
of irreducible representations and intersections in the geomety of the
Grassmannian variety.  Schur functions are the
generating functions of tableaux.  A semi-standard
tableau  of weight $\mu=(\mu_1,\ldots,\mu_r)$ is a nested 
sequence of partitions 
\begin{equation}
\label{tab}
\emptyset=
\lambda^{(0)}\subset\lambda^{(1)}\subset\cdots \subset\lambda^{(r)}
\end{equation}
such that
$\lambda^{(i)}/\lambda^{(i-1)}$ is a horizontal $\mu_i$-strip.
It is generally represented with a filling of shape $\lambda^{(r)}$
by placing $i$ in the cells of the skew $\lambda^{(i)}/\lambda^{(i-1)}$.
When the weight of a tableau is $(1,1,\ldots,1)$
it is called {\it standard}.
$SSYT(\lambda,\mu)$ denotes 
the set of semi-standard tableaux of shape $\lambda$ and weight $\mu$
and the union over all weights is $SSYT(\lambda)$.
For any partition $\lambda$, the Schur function is
$$
s_\lambda = \sum_{T\in SSYT(\lambda)} x^{\weight(T)}\,.
$$

Refinements of the Schur basis for $\Lambda$
to bases for $\Lambda^{(n)}$ and $\Lambda_{(n)}$ give a 
combinatorial framework that can be applied to the cohomology 
and homology of $\Gr$.  Let $k=n-1$ throughout.
The basis of $k$-Schur functions for $\Lambda_{(n)}$ was
introduced in \cite{LMcore}, inspired by the Macdonald polynomial 
study of \cite{LLM} summarized in the introduction.
The basis 
for $\Lambda^{(n)}$ that is dual to the $k$-Schur basis with 
respect to the Hall-inner product
arose in the context of the quantum cohomology of Grassmannians 
in \cite{[LMhecke]}.  
Appealing to the algebraic 
nil-Hecke ring construction of Kostant and Kumar~\cite{KK1} and the work 
of Peterson~\cite{Pet}, 
Lam~\cite{Lam} proved that the Schubert classes $\xi_w$ and $\xi^w$ 
can be represented explicitly by the $k$-Schur functions
in $\Lambda_{(n)}$ and $\Lambda^{(n)}$, respectively. 

For our purposes, we define the $k$-Schur functions of $H^*(\Gr)$ as
the weight generating functions of a combinatorial object called {\it affine 
factorizations} and then introduce the homology $k$-Schur functions by duality.  
For any composition $\alpha\in \mathbb N^\ell$ with parts smaller than $n$ 
and $w\in \tilde S_n$ of length $|\alpha|$, an affine factorization 
for $w$ of weight $\alpha$ is a decomposition 
$$
w = v^{\ell}\cdots v^{1}\,,
$$
where $v^{i}$ is a cyclically decreasing permutation of length $\alpha_i$.
The representatives for the Schubert classes of $H^*(\Gr)$
are then defined, for $\lambda\in\Co$, by
\begin{equation}
\label{dualkschur}
\mathfrak S_\lambda^{(n)} = 
\sum_{w_\lambda=v^r\cdots v^1} x_1^{\ell(v^1)}\cdots x_r^{\ell(v^r)}\,,  
\end{equation}
over all affine factorizations $v^r\cdots v^1$ of $w$
(In \cite{LamStan}, 
by dropping the condition that $w_\lambda$ is affine Grassmannian,
these are extended to a more general family of 
functions that relate to the stable limits of Schubert polynomials
\cite{LSschub,[St]}).
The set $\{\mathfrak S_\lambda^{(n)}\}_{\lambda\in\Co}$ is a basis for $\Lambda^{(n)}$
and we take the $k$-Schur representatives for Schubert classes of $H_*(\Gr)$
to be the dual basis $\{s_\nu^{(n)}\}_{\nu\in \Co}$
with respect to the Hall-inner product.
That is, the $k$-Schur functions are defined by the relation
\begin{equation}
\label{def:kschur}
\langle \mathfrak S_{\lambda}^{(n)},s_\nu^{(n)}\rangle =\delta_{\lambda\nu}\,.
\end{equation}

\subsection{Affine Bruhat countertableaux}

Here we derive a new combinatorial object with which to study 
$H^*(\Gr)$ and $H_*(\Gr)$ by considering the 
association between cyclically decreasing permutations
and horizontal strong strips that was made in Section~\ref{sec:pieri}.
Recall that the sequence \eqref{tab} can be represented by 
its \textit{countertableau} filling,
derived by placing an $r+1-i$ in $\lambda^{(i)}/\lambda^{(i-1)}$.

\begin{definition}
\label{def:abc}
Fix composition $\alpha=(\alpha_1,\ldots,\alpha_r)$ 
with $\alpha_i<n$ and $n$-core $\lambda^{(r)}$ 
of degree $|\alpha|$.  An \textit{affine Bruhat countertableau} 
of shape $\lambda^{(r)}$ and weight $\alpha$ is a skew tableau
$\lambda^{(r)}=\mu^{(0)}\subset\cdots\subset\mu^{(r)}$
such that 
\begin{equation}
\label{abc2weak}
\mu^{(x)}=(\mu^{(1)}_1,\ldots,\mu^{(x-1)}_{x-1},
\lambda_1^{(r-x)}+n-1,\lambda^{(r-x)})\,,
\end{equation}
where $(\lambda^{(x-1)},\lambda^{(x)})$ is
a horizontal strong $(n-1-\alpha_x)$-strip 
for $1\leq x\leq r$ and $\lambda^{(0)}=\emptyset$.
%
\end{definition}

An affine Bruhat countertableau (or $ABC$) is represented by 
its skew countertableau filling where $r-x+1$ is placed in 
the cells of $\mu^{(x)}/\mu^{(x-1)}$.
We denote the set of $ABC$'s of shape $\lambda$ and weight $\alpha$ 
by $ABC(\lambda,\alpha)$ and let $ABC(\lambda)$ be their union 
over all weights $\alpha$.

\begin{example}
\label{ex1abc} 
For $n=6$, $\mu^{(0)}=(4,3,0)\subset (9,4,2)\subset (9,8,3)
\subset (9,8,5)=\mu^{(3)}$ is an $ABC$ of shape $(4,3)$ and 
weight $(3,3,1)$, represented by its countertableau filling
\begin{displaymath} A =
\text{\footnotesize
\tableau*[scY]{3&3&2&1&1\cr&&&3&
2& 2& 2&2 
\cr &&&&3&3& 3& 3&3
}}
\,.
\end{displaymath}
The corresponding strong strips are
\begin{displaymath}
\lambda^{(3)}=
\text{\tiny\tableau*[sbY]{&&\cr&&&}}
\lessdot_B
\text{\tiny\tableau*[sbY]{3\cr&&\cr &&&&3}}
\lessdot_B
\text{\tiny\tableau*[sbY]{3&3\cr&&\cr&&&&3&3}}
\lessdot_B
\text{\tiny\tableau*[sbY]{3&3\cr&&\cr&&&&3&3& \fr[l,t,b] 3 & \fr[t,b,r] 3}}
\lessdot_B
\text{\tiny\tableau*[sbY]{3&3\cr&&&3\cr&&&&3&3& \fr[l,t,b] 3 & \fr[t,b,r]
3 & 3}}.
=(\lambda_1^{(2)}+5,\lambda^{(2)})
\,.
\end{displaymath}
\begin{displaymath}
\lambda^{(2)}=
\text{\tiny\tableau*[sbY]{&\cr&&&}} \lessdot_B
\text{\tiny\tableau*[sbY]{&\cr&&&& \fr[l,t,b] 2 & \fr[t,b] 2 & \fr[t,b,r] 2}}
\lessdot_B
\text{\tiny\tableau*[sbY]{&&2\cr&&&& \fr[l,t,b] 2 & \fr[t,b] 2 &
\fr[t,b,r] 2 &2}}
=(\lambda_1^{(1)}+5,\lambda^{(1)})
\,.
\end{displaymath}
\begin{displaymath}
\lambda^{(1)}=
\text{\tiny\tableau*[sbY]{&&}} \lessdot_B \text{\tiny\tableau*[sbY]{&&&1}}
\lessdot_B \text{\tiny\tableau*[sbY]{&&&1&1}}
=(\lambda^{(0)}+5)
\,.
\end{displaymath}
\end{example}

\begin{lemma}
\label{lem:abcseq}
Relation \eqref{abc2weak} uniquely identifies
the element $\mu^{(0)}\subset\cdots\subset\mu^{(r)}$ in
$ABC(\mu^{(0)},\alpha)$ with the sequence
$\emptyset=\lambda^{(0)} \subset \lambda^{(1)} \subset \cdots \subset
\lambda^{(r)}=\mu^{(0)}$
where $(\lambda^{(p-1)},\lambda^{(p)})$ is a horizontal 
strong $(n-1-\alpha_{p})$-strip,
\end{lemma}

\begin{proof}
The forward direction is immediate from the definition of $ABC$.
On the other hand, consider a sequence
$\emptyset=\lambda^{(0)} \subset \lambda^{(1)} \subset \cdots \subset
\lambda^{(r)}$ where $(\lambda^{(p-1)},\lambda^{(p)})$ is a horizontal 
strong $(n-1-\alpha_{p})$-strip for $1\leq p \leq r$.
Let $\mu^{(0)}=\lambda^{(r)}$, and for $1\leq p\leq r$, define
$\mu^{(p)}$ by Relation~\eqref{abc2weak}.
From this, we find that
$$
\mu^{(p)}=(\lambda_1^{(r-1)}+n-1,\ldots,
\lambda_1^{(r-p+1)}+n-1,R(n-1,\lambda^{(r-p)}))
\,,
$$
and need only to confirm that $\mu^{(p)}/\mu^{(p-1)}$ is a horizontal strip
for $1\leq p\leq r$.  Note that
$$
\mu^{(p-1)}=(\lambda_1^{(r-1)}+n-1,\ldots,
\lambda_1^{(r-p+1)}+n-1,\lambda^{(r-p+1)})\,.
$$
The claim follows by recalling that when
$(\lambda^{(r-p)},\lambda^{(r-p+1)})$ is a horizontal strong strip,
$R(n-1,\lambda^{(r-p)})/\lambda^{(r-p+1)}$ is a horizontal strip.
\end{proof}

\begin{theorem}
\label{dualkSchurABC}
For any $n$-core $\lambda$,
\begin{equation}
\label{mainresultindual}
\mathfrak{S}_{\lambda}^{(n)} =
\sum_{A\in ABC(\lambda)}
x^{\text{weight}(A)}
\,.
\end{equation}
\end{theorem}

\begin{proof}
Fix a composition $\alpha$ of length $r$ with parts smaller than $n$ 
and an $n$-core $\lambda$ such that $deg(\lambda)=|\alpha|$.  
Define a map on domain $ABC(\lambda,\alpha)$ by
sending
$$
\Theta: A \longmapsto v^r\cdots v^1
\quad\text{where $v^i=w_{\lambda^{(i)}}w_{\lambda^{(i-1)}}^{-1}$}\,,
$$
for the unique sequence 
$\lambda^{(0)}\subset\cdots \subset\lambda^{(r)}$ 
associated to $A$ via Lemma~\ref{lem:abcseq}.
We claim that $\Theta$ is a bijection whose image is the
set of affine factorizations of $w_\lambda$ of weight $\alpha$.
For this, we must prove that
$v^{i}$ is a cyclically decreasing permutation of length $\alpha_i$, 
$\ell(w_\lambda)=|\alpha|$ and that $\Theta$ is a bijection.  

When $\alpha=(\alpha_1)$ has only one part,
the unique element of $ABC(\lambda,\alpha)$
corresponds to the sequence $\emptyset\subset\lambda$ where $\lambda=(\alpha_1)$.
Its image under $\Theta$ is $v^1=w_{(\alpha_1)}=s_{\alpha_1-1}\cdots s_0$,
the only decomposition of $w_\lambda$ into one cyclically decreasing permutation
$v^1$ of length $\alpha_1$.  
By induction, assume that the sequence $\emptyset=\lambda^{(0)} 
\subset \cdots \subset \lambda^{(r-1)}$,
where $(\lambda^{(j-1)},\lambda^{(j)})$ is a horizontal
strong $(n-1-\alpha_j)$-strip,
corresponds uniquely to a decomposition of 
$w_{\lambda^{(r-1)}} = v^{r-1}\cdots v^1$
into cyclically decreasing permutations $v^j$ of length
$\alpha_j$, for $j<r$.
Since
$v^r=w_{\lambda^{(r)}}w_{\lambda^{(r-1)}}^{-1}$ is cyclically decreasing of length 
$\alpha_r$ if and only if $(\lambda^{(r-1)},\lambda^{(r)})$ is a horizontal strong
$(n-1-\alpha_r)$-strip by Theorem~\ref{prop:main}, the result
follows by induction.
\end{proof}

Because $\mathfrak S_\lambda^{(n)}$ is a symmetric
function, the coefficient of $x^\alpha$ in \eqref{mainresultindual}
equals the coefficient of $m_\mu$, where $\mu$ is the
non-increasing rearrangement of the parts of $\alpha$.
The set of monomial symmetric functions indexed by elements in 
$\mathcal P^n = \{ \lambda \in \mathcal P : \lambda_1 < n \}$ is
a basis for $\Lambda^{(n)}$ and in fact,
the transition matrix from $\{\mathfrak S_\lambda^{(n)}\}_{\lambda
\in\Co}$ to $\{m_\mu\}_{\mu\in\mathcal P^n}$ is unitriangular.
The unitriangularity relation is described by an identification 
from \cite{LMcore} of 
$n$-cores with partitions of $\mathcal P^n$, given by
the map 
$$\mfc: \mathcal P^n \longrightarrow \Co\,,$$ 
where $\mfc^{-1}(\gamma)=(\lambda_1,\ldots,\lambda_\ell)$ 
and $\lambda_i$ is the 
number of cells in row $i$ of $\gamma$ with hook-length
smaller than $n$.
The unitriangularity relation is taken with
respect to the dominance order on partitions of the same degree, 
where $\lambda\lhd\mu$ 
when $\lambda_1+\cdots+\lambda_s\leq \mu_1+\cdots+\mu_s$ for
all $s$.  It was proven in \cite{LMcore,LamStan} that for 
any $\lambda\in\mathcal P^n$,
\begin{equation}
\label{uni}
\mathfrak S_{\core(\lambda)}^{(n)} = 
m_{\lambda}+ \sum_{\substack{\mu\in\mathcal P^n \\\
\mu\,\, \not\!\!\rhd \lambda}} K_{\lambda\mu}^n\,m_\mu
\,.
\end{equation}

\begin{corollary}
\label{cor:uni}
For $\lambda,\mu\in\mathcal P^n$,
$K_{\lambda\mu}^{n}$ is the number of
$ABC$'s of shape $\core(\lambda)$ and weight $\mu$.
In particular, there is a unique
$ABC$ of shape $\core(\lambda)$ and weight $\lambda$
and $ABC(\core(\lambda),\mu)=\emptyset$ when
$\mu\not\!\!\lhd\lambda$.
\end{corollary}


\begin{proposition}
\label{prop:abckinf}
\cite{[LMproofs]}
For an $n$-core $\lambda$ where $deg(\lambda)<n$,
we have $\mathfrak S_\lambda^{(n)}= s_\lambda$.
\end{proposition}
\begin{proof}
Given an $n$-core $\lambda$ where $deg(\lambda)<n$
and a partition $\mu$ of length $r$ where
$|\mu|=deg(\lambda)$, we shall
prove that there is a bijection between 
$ABC(\lambda,\mu)$ and $SSYT(\lambda,\mu)$.

By Lemma~\ref{lem:abcseq}, $A\in ABC(\lambda,\mu)$ is 
defined by a sequence of $n$-cores
\begin{equation}
\label{lambdaseq}
\emptyset\subset\lambda^{(1)}\subset\cdots\subset\lambda^{(r)}=\lambda
\end{equation}
where $(\lambda^{(i-1)},\lambda^{(i)})$ is a horizontal strong 
$(n-1-\mu_i)$-strip.   Note in particular that
$n-1-\mu_i=n-1+deg(\lambda^{(i-1)})-deg(\lambda^{(i)})$.
Since Theorem~\ref{prop:main} implies that
$w_{\lambda^{(i)}}w_{\lambda^{(i-1)}}^{-1}$ is cyclically decreasing
of length $\ell(w_{\lambda^{(i)}})-\ell(w_{\lambda^{(i-1)}})=\mu_i$,
there are $\mu_i$ distinct residues 
labelling the cells of the horizontal strip $\lambda^{(i)}/\lambda^{(i-1)}$
by Lemma~\ref{lem:cychor}.
If $\lambda\in\Co$ with $deg(\lambda)<n$, then no two cells that
lie at the top of their column in $\lambda^{(i)}\subset\lambda$ 
can have the same $n$-residue.  Therefore, 
$\lambda^{(i)}/\lambda^{(i-1)}$ is a horizontal $\mu_i$-strip,
implying that \eqref{lambdaseq} is an element of $SSYT(\lambda,\mu)$.

On the other hand, given a semi-standard tableau $T\in SSYT(\lambda,\mu)$
defined by \eqref{lambdaseq}, it suffices to show that
$(\lambda^{(i-1)},\lambda^{(i)})$ is a horizontal strong 
$n-1-\mu_i$-strip for all $i$.   By definition of semi-standard tableau,
$\lambda^{(i)}/\lambda^{(i-1)}$ is a horizontal $\mu_i$-strip for $1 \leq
i \leq r$.  Since $deg(\lambda) < n$, $\mu_i$ distinct residues label the 
cells of $\lambda^{(i)}/\lambda^{(i-1)}$ and $deg(\lambda^{(i)})-deg(\lambda^{(i-1)})
=\mu_i$.  Therefore, by Lemma~\ref{lem:cychor},
$w_{\lambda^{(i)}}w_{\lambda^{(i-1)}}^{-1}$ is
cyclically decreasing with length $\mu_i$.  
Theorem \ref{prop:main} then implies the result.
\end{proof}

An $ABC$ countertableau $A$ comes equipped with a ribbon tiling
specified by its defining strong strips.
Let the {\it column residue} of every cell in column $c$ 
be $c-1\pmod n$.  Recall that the bijection $\phi$ identifies
a horizontal strong $(n-1-m)$-strip $(\lambda^{(i)},\lambda^{(i+1)})$ 
with a reduced word $j_1\cdots j_m$ for
$w_{\lambda^{(i+1)}}w_{\lambda^{(i)}}^{-1}$
by a tiling of $R(n-1,\lambda^{(i)})/\lambda^{(i+1)}$ with
ribbons of height one that are determined by placing a tail
in the rightmost cell of the bottom row with 
residue $t \in \{x-1,\ldots,x+1\}\backslash\{j_1,\ldots,j_{m}\}$
for $x = \lambda_1^{(i)}-1 \pmod n$.
Therefore, since $A$ is given by $\mu^{(0)}\subset\cdots \subset \mu^{(r)}$
where
$\mu^{(r+1-i)}=(\mu_1^{(r-i)},\ldots,\mu_{r-i}^{(r-i)},R(n-1,\lambda^{(i)}))$,
these ribbons tile row $r+1-i$ of $A$ where the residues of their tails are
now column residues and ribbons containing letter $j>i$ in row $i$ 
are copies of the ribbons specified in row $j$.

A diagram derived from an $ABC$ $A$ of weight $\mu\in\mathcal P^n$
called the {\it extension}  $ext(A)$ is a useful tool to convert
between affine factorizations and $ABC$'s.  It is formed by appending a ribbon 
of length $\lambda^{(x)}_1-\lambda^{(x-1)}_1+1$ to the end of row $x$,
and then deleting any letter larger than $x$ in row $x$
and the tail of every ribbon containing $x$. 

\begin{example}\label{ssExtABC}
For $n=6$ and an $ABC$ $A$ of weight $(3,3,3,1)$, we have
$$
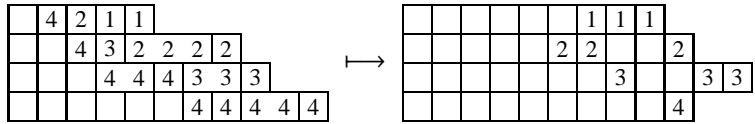
\text{\footnotesize{
\tableau*[scY]{ &4&2&1&1 \cr &&4&3& \fr[l,t,b] 2& \fr[t,b,] 2
& \fr[t,b,r] 2&2 \cr & & & \fr[l,t,b] 4& \fr[t,b,r] 4&4& \fr[l,t,b] 3& \fr[t,b,r] 3&3 \cr & & & & & &4&4& \fr[l,t,b] 4& \fr[t,b,r] 4&4}}}
\; \longmapsto \;
\text{\footnotesize{\tableau*[scY]{ &&&&& & 1& 1& 1\cr  
&&&& &  2&  2&& &  2\cr
 & & &&  &&& 3&& &  3& 3\cr  
& & & & & &&&  & 4}}}
$$
\end{example}


\begin{lemma}
\label{lem:abcres}
For $A\in ABC(\lambda)$, $j_1\cdots j_\ell$ is a reduced word for
$v^i$ in the affine factorization $\Theta(A)=v^r\cdots v^1$ if 
and only if the cells containing $i$ in 
$ext(A)$ have column residues $\{j_1,\ldots,j_\ell\}$.
\end{lemma}
\begin{proof}
Consider the $ABC$ given by $\mu^{(0)}\subset\cdots\subset\mu^{(r)}$
where
$$
\mu^{(r-i+1)}=(\lambda_1^{(r-1)}+n-1,\ldots,
\lambda_1^{(i)}+n-1,R(n-1,\lambda^{(i-1)}))
$$
and
$$
\mu^{(r-i)}=(\lambda_1^{(r-1)}+n-1,\ldots,
\lambda_1^{(i)}+n-1,\lambda^{(i)})\,.
$$
The parts of $\mu^{(r-i+1)}$ and $\mu^{(r-i)}$ differ only in 
the top $i$ rows where we find the cells of $A$ containing an $i$
forming the horizontal strip $D=R(n-1,\lambda^{(i-1)})/\lambda^{(i)}$.
In particular, the bottom row of $D$ in $A$ has an $i$ in columns 
$\lambda_1^{(i)}+1,\ldots,\lambda_1^{(i-1)}+n-1$.
To determine which cells of $ext(A)$ contain $i$,
a ribbon of length $\lambda^{(i)}-\lambda^{(i-1)}+1$ is appended
to the end of this row and we must appeal to the strong strips 
to delete the tails.

Theorem~\ref{dualkSchurABC} implies that $\Theta(A)=v^r\cdots v^1$ is an
affine factorization for $w_\lambda$ where
$v^i=w_{\lambda^{(i)}}w_{\lambda^{(i-1)}}^{-1}$.
For each $i$, the proof of Theorem~\ref{mainresultindual}
uniquely identifies a reduced word $j_1\cdots j_{n-1-m}$ for $v^i$
with the strong chain 
$$
\lambda^{(i)}=\nu^{(0)}\lessdot_B\cdots\lessdot_B
\nu^{(m)} =R(n-1,\lambda^{(i-1)})
\,,
$$
where $\nu^{(i-1)}$ is obtained from $\nu^{(i)}$ by deleting 
all removable copies of the ribbon whose tail has residue $a_{m-i+1}$
and lies in the bottom row, for the elements
$x+1 \leq a_m < \cdots < a_1 \leq x-1$ 
of $\{x-1,\ldots,x+1\}\backslash\{j_1,\ldots,j_{n-1-m}\}$
and $x=\lambda^{(i-1)}-1\pmod n$.  Appending a ribbon of 
length $\lambda_1^{(i)}-\lambda^{(i-1)}_1+1$ 
to the end of the bottom row of $R(n-1,\lambda^{(i-1)})/\lambda^{(i)}$
gives a skew shape where $\{j_1,\ldots,j_{n-1-m}\}$ is the set of
column residues labeling the cells in the bottom row excluding 
ribbon tails.
\end{proof}

\section{$t$-generalized affine Schubert polynomials}
\label{sec:mac}

Macdonald's basis of $P$-functions 
(see \cite{Macbook}) are defined by
\begin{equation}
\label{Pdef}
P_\lambda(x;t) =\frac{1}{v_\lambda(t) }
\sum_{w\in S_n} w( x^\lambda \prod_{i<j} \frac{x_i-tx_j}{x_i-x_j})
\end{equation}
where $v_\lambda(t)=\prod_{j\geq 0} \prod_{i=1}^{m_j}\frac{1-t^i}{1-t}$ 
for $m_j$ the multiplicity of $j$ in $\lambda$.
For convenience, we work with the deformation
$\tilde P_\lambda(x;t)=t^{-n(\lambda)}P_\lambda(x;t^{-1})$
where $n(\lambda)=\sum_i (i-1)\lambda_i$.
The set of $\tilde P$-functions forms a basis for $\Lambda$ that
generalizes the monomial basis; when $t=1$, $\tilde P_\mu(x;1)=m_\mu$.
One of the most important features of this basis is that
the Kostka-Foulkes polynomials are 
inscribed in the Schur to $\tilde P$-function transition matrix:
\begin{equation}
\label{sintP}
s_\lambda=\sum_\mu K_{\lambda\mu}(t)\,\tilde P_\mu(x;t)\,.
\end{equation}
Moreover, the $q=0$ case of the Macdonald polynomials
$\{H_\mu(x;0,t)\}$
arise as the the dual basis to $\{\tilde P_\mu\}$ with respect
to the Hall-inner product.

Lascoux and Sch\"utzenberger found an intrinsically positive
formula for the Kostka-Foulkes polynomials
by associating
a statistic (non-negative integer) called $\cocharge$ to 
each semi-standard tableaux and proving that
\begin{equation}
\label{kostkafoulkes}
K_{\lambda\mu}(t)=\sum_{T\in SSYT(\lambda,\mu)}
t^{\cocharge(T)}\,.
\end{equation}
The $\cocharge$ of a standard tableau $T$ is the sum of the 
entries in the {\it index vector} 
$I(T)=[0,I_2,\ldots,I_m]$ which is defined by setting
$I_{r}=I_{r-1}$ when the content of $r$ is larger than
the content of $r-1$ and otherwise setting
$I_{r}=I_{r-1}+1$.
%
The notion is extended to give the $\cocharge$ of a semi-standard
tableau with generic weight by successively computing the index 
of an appropriate subset of $i$ cells containing the letters $1,2,\ldots,i$.

\begin{definition}
\label{convchoice}
From a specific $x$ in cell $c$ of a tableau $T$,
the desired choice of $x+1$ is the south-easternmost
one lying above $c$.  If there are none above $c$,
the choice is the south-easternmost $x+1$ in all of $T$.
\end{definition}

Consider now any semi-standard tableau $T$  with partition weight.
Starting from the rightmost 1 in $T$, use Definition~\ref{convchoice}
to distinguish a standard sequence of $i$ cells containing $1,2,\ldots,i$.
Compute the index and then delete all cells in this sequence.
Repeat the process on the remaining cells.
The total {\it cocharge} is defined to be the sum of all the index vectors.

\begin{example}\label{ex:convcharge}
The $\cocharge$ of the following tableau is 25:
$$
\text{\footnotesize{
\tableau[sbY]
{\tf 6\cr\tf 4&5 \cr\tf 3&4
\cr2&\tf 2&3&\tf 5\cr1&1&\tf 1&2&3&\tf 7\cr}}}
\qquad
\qquad
\text{\footnotesize{
\tableau[sbY]
{\cr&\tf 5 \cr&\tf 4
\cr\tf 2&&3&\cr1&\tf 1&& 2&\tf 3&\cr}}}
\qquad
\qquad
\text{\footnotesize{
\tableau[sbY]
{\cr& \cr&
\cr&&\tf 3&\cr \tf 1&&& \tf 2&&\cr}}}
\qquad
\qquad
$$
$$
\hspace{-2.5cm}
I=[0,1,2,3,3,4,4]\qquad
\quad
I=[0,1,1,2,3]\qquad
\qquad\quad
I=[0,0,1]
$$
\end{example}

The $\ncocharge$ of an $ABC$ depends on computing an index vector
in a similar spirit.  However, the role of $n$ brings forth an additional concept.
Any ribbon in an $ABC$ that is filled with letter $i$ but does 
not lie in row $i$ is called an {\it offset}.  The number of cells 
that are not tails in all the offsets,
$$
off(A) =
\sum_{\substack{\text{R: offset in }A}}\hspace{-0.15in}(\text{size}(R) - 1)\,,
$$
is one of the two components needed when computing $\ncocharge$.
The index is the other, computed on the extension of $A$.
Our construction of the index considers only the cells 
in $ext(A)$.  For $A$ of weight $1^m$, $ext(A)$ is standard;
there is exactly one cell in each row $i$ (coming from the 
single ribbon head with an $i$ in row $i$ of $A$).  
In this case, the $\kcocharge$ is defined by computing an index
vector $I=[0,I_2,\ldots,I_m]$ defined by
$$
I_{r+1} = \begin{cases}
I_r & \text{ when $r+1$ is east of $r$}\\
I_r+1 & \text{ when $r+1$ is west of $r$}\,.
\end{cases}
$$

\begin{example}
From the $ABC$ of weight $(1^7)$:
$$
\text{\footnotesize{
A = \tableau*[scY]{2&1&1 \cr 5&3&2&2 \cr  &4& \fr[l,t,b] 3& \fr[t,b,r] 3&3 \cr  &6&5&4&4 \cr & \fr[l,t,b] 7& \fr[t,b,r] 7& \fr[l,t,b] 5& \fr[t,b,r] 5&5 \cr  & & &6& \fr[l,t,b] 6& \fr[t,b,r] 6 \cr & & & \fr[l,t,b] 7& \fr[t,b,r] 7&7}}}
\hspace{0.1in}
\;
\longmapsto
\;
\text{\footnotesize{
ext(A) = \tableau*[scY]{&&& &  1\cr &&&& & 2 \cr
&& & 3\cr &&&&&&4 \cr 
& & & & 5 \cr  & & && & 6 \cr
& & & &  7}}}
\hspace{0.1in}
\longmapsto
\hspace{0.1in}
I=[0,0,1,1,2,2,3]
$$
\end{example}
Equipped with a method to obtain the index when $ext(A)$
has a single $i$ in row $i$, we describe a method for 
extracting standard fillings
from an $ABC$ of arbitrary weight $\mu\in\mathcal P^n$.

\begin{definition}\label{wordABC}
\label{wierdchoice}
Given an $ABC$ $A$ of weight $\mu\in\mathcal P^n$, consider
its labelling by column residues.  Iteratively earmark a standard 
sequence starting with the rightmost 1. From an $x$ (of column residue $i$)
the appropriate choice of $x+1$ will be
determined by choosing its column residue from the set 
$\mathcal B$ of all column residues labelling the $x+1$'s.  
Reading counter-clockwise from $i$, this choice is the 
closest $j\in \mathcal B$ on a circle labelled clockwise
with $0,1,\ldots,n-1$.
\end{definition}

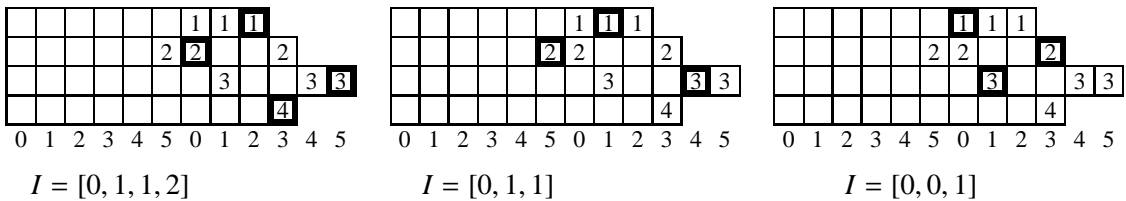
\begin{example}
$$
\text{\footnotesize{
\tableau*[scY]{ &&&&& & 1& 1&\tf 1\cr  
&&&& &  2& \tf 2&& &  2\cr
 & & &&  &&& 3&& &  3& \tf 3\cr  
& & & & & &&&  &\tf 4\cr
\bl 0&\bl 1&\bl 2&\bl 3&\bl 4&\bl 5&\bl 0&\bl 1&\bl 2&\bl 3&\bl 4&\bl 5}}}
\quad
\text{\footnotesize{
\tableau*[scY]{ &&&&& & 1&\tf 1& 1\cr  
&&&& & \tf 2&  2&& &  2\cr
 & & &&  &&& 3&& &  \tf 3& 3\cr  
& & & & & &&&  & 4\cr
\bl 0&\bl 1&\bl 2&\bl 3&\bl 4&\bl 5&\bl 0&\bl 1&\bl 2&\bl 3&\bl 4&\bl 5}}}
\quad
\text{\footnotesize{
\tableau*[scY]{ &&&&& &\tf 1& 1& 1\cr  
&&&& &  2&  2&& & \tf 2\cr
 & & &&  &&& \tf 3&& &  3& 3\cr  
& & & & & &&&  & 4\cr
\bl 0&\bl 1&\bl 2&\bl 3&\bl 4&\bl 5&\bl 0&\bl 1&\bl 2&\bl 3&\bl 4&\bl 5}}}
$$
$$
I=[0,1,1,2]
\qquad\qquad \qquad\qquad 
I=[0,1,1]
\qquad \qquad\qquad 
\qquad\qquad 
I=[0,0,1]
\qquad\qquad 
$$
\end{example}

\begin{definition}
For an $ABC$ $A$ of weight $\mu\in\mathcal P^n$, the
$n$-cocharge of $A$ is defined by
$$
\ncocharge(A)= \sum_{r}I_r(A) + off(A)\,.
$$
\end{definition}

We use the Schur expansion in $\tilde P$-functions \eqref{sintP}
as a guide to introduce a new family of symmetric
functions involving the parameter $t$ that play a 
role in affine Schubert calculus and the theory of
Macdonald polynomials.

\begin{definition}
For $\lambda\in\mathcal P^n$, set
\begin{equation}
\label{tinvdef}
\mathfrak S_{\core(\lambda)}^{(n)}(x;t) = 
\sum_{\mu\in \mathcal P^n} 
K_{\lambda\mu}^{n}(t)\, \tilde P_\mu(x;t)
\,,
\end{equation}
where the coefficients are taken to be the
$\ncocharge$ generating functions of $ABC$'s
(or weak Koskta-Foulkes polynomials),
\begin{equation}
\label{genKFPoly}
K_{\lambda\mu}^{n}(t) = \sum_{A\in ABC(\core(\lambda),\mu)}
 t^{\ncocharge(A)}\,.
\end{equation}
\end{definition}

For each $n>1$, consider the restricted linear span of Macdonald polynomials
$H_\lambda(x;0,t)$ and the $\tilde P$-functions defined by
$$
\Lambda_{(n)}^t = \mathcal L\{H_\lambda(x;0,t) : \lambda_1<n\}
\quad\text{and}\quad
\Lambda^{{(n)}^t} = \mathcal L\{\tilde P_\lambda(x;t) : \lambda_1<n\}\,.
$$
When $t=1$,
these reduce to $\Lambda_{(n)}$ and $\Lambda^{(n)}$, respectively.

\begin{proposition} \label{dualtbasis}
For $n>1$, the set $\{\mathfrak S_\lambda^{(n)}(x;t)\}_{\lambda\in\Co}$
forms a basis for $\Lambda^{{(n)}^t}$ that reduces to a set of 
representatives for the Schubert cohomology classes of $\Gr$ when $t=1$.
\end{proposition}
\begin{proof}
Since $\tilde P_\mu(x;1)=m_\mu$, we have that
$\{\mathfrak S_\lambda^{(n)}(x;1)\}=\{\mathfrak S_\lambda^{(n)}\}$ 
which gives a set of Schubert representatives $\{\xi^{w_\lambda}\}$ 
in  $H^*(\Gr)$ by Theorem~\ref{dualkSchurABC}.
The set $\{\tilde P_\lambda(x;t): \lambda\in\mathcal P^n\}$ is linearly
independent and therefore a basis for $\Lambda^{{(n)}^t}$. 
Therefore, $\mathfrak S_\lambda^{(n)}(x;t)\in\Lambda^{{(n)}^t}$
implies that 
it suffices to show the transition matrix between
$\tilde P$-functions and $\{\mathfrak S_\lambda^{(n)}(x;t)\}_{\lambda\in\Co}$
is invertible.  The matrix is square since there is a bijection 
between $n$-cores and elements of $\mathcal P^n$ and in fact invertible
since Corollary~\ref{cor:uni} implies
$$
K_{\eta\mu}^n(t) = 
t^{\ncocharge(B)} + 
\sum_{\substack{A\in ABC(\core(\eta),\mu) \\\ \eta\,\,\not\lhd\mu}}
t^{\ncocharge(A)}
$$
where $B$ is the unique $ABC$ of wieght $\eta$ and shape $\core(\eta)$.
\end{proof}


Let the set of functions $\{s_\gamma^{(n)}(x;t)\}_{\gamma\in\Co}$ 
be the basis for $\Lambda_{(n)}^t$ defined by the duality relation,
with respect to the Hall-inner product,
\begin{equation}
\label{eq:innerwt}
\langle \mathfrak S_\mu^{(n)}(x;t),s_\gamma^{(n)}(x;t)\rangle
=\delta_{\mu\gamma}
\,.
\end{equation}
Since the Macdonald polynomial $H_\lambda(x;q,t)$ reduces to
the Hall-Littlewood polynomial $H_\lambda(x;t)$ when $q=0$,
we are now able to prove there is a natural tie between affine
Schubert calculus and Macdonald polynomials.

\begin{corollary}
For each $n$-core $\lambda$, $s_\lambda^{(n)}(x;1)$ represents the
Schubert class $\xi_{w_\lambda}$ in $H_*(\Gr)$ and
for every $\mu\in\mathcal P^n$, the Macdonald polynomial
at $q=0$ satisfies the non-negative expansion
\begin{displaymath}
H_\mu(x;0,t)=
\sum_{\lambda} K_{\lambda\mu}^{n}(t)\, s_{\core(\lambda)}^{(n)}(x;t)
\,.
\end{displaymath}
\end{corollary}
\begin{proof}
The result follows by noting that
\eqref{eq:innerwt} reduces to \eqref{def:kschur} when $t=1$
by Proposition~\ref{dualtbasis} and that \eqref{def:kschur} defines
the $k$-Schur functions representing the Schubert class 
$\xi_{w_\lambda}$.
\end{proof}

Combinatorial results on $ABC$'s can be used to prove that
when $n>|\lambda|$, 
both $s_\lambda^{(n)}(x;t)$ and $\mathfrak S_\lambda^{(n)}(x;t)$ 
reduce to the Schur function  $s_\lambda$.

\begin{lemma}
Given a partition $\lambda$ where $|\lambda|<n$,
if $T$ is a semi-standard tableau of shape $\lambda$ then
Definition~\ref{wierdchoice} reduces 
to Definition~\ref{convchoice}.
\end{lemma}
\begin{proof}
Consider $x$ of $n$-residue $i$ in $T$.  Note that $|\lambda|<n$ 
implies that there is a unique cell $c$ of residue $i$ that contains 
$x$.  Let $j$ be the first entry on the circle
reading counter-clockwise from $i$ that is a residue
of a cell containing $x+1$.
If there is an $x+1$ above $c$, then
the south-easternmost cell containing
an $x+1$ that is above $c$ has residue $j$
since there are no $x+1$'s of a residue counter-clockwise
between $i$ and $j$.  If there are none above $c$,
then for the same reason, the south-easternmost cell
containing an $x+1$ has residue $j$.
\end{proof}

\begin{proposition}
For $\lambda\in\Co$ with $deg(\lambda)<n$, 
$\mathfrak S_\lambda^{(n)}(x;t)=s_\lambda^{(n)}(x;t)=s_\lambda$.
\end{proposition}
\begin{proof}
Let $\lambda\in\Co$ with $deg(\lambda)<n$. 
By Lemma~\ref{lem:abcseq}, $A\in ABC(\lambda,\mu)$ is 
defined by a sequence of $n$-cores
\begin{equation}
\label{lambdaseq1}
\emptyset\subset\lambda^{(1)}\subset\cdots\subset\lambda^{(r)}=\lambda
\end{equation}
where $(\lambda^{(i-1)},\lambda^{(i)})$ is a horizontal strong 
$(n-1-\mu_i)$-strip.   
This sequence is in one to one correspondence with
a unique element $T \in SSYT(\lambda,\mu)$ by
Proposition~\ref{prop:abckinf}. 
Lemmas~\ref{lem:cychor} and \ref{lem:abcres}
imply that the set of residues 
labelling cells of $\lambda^{(i)}/\lambda^{(i-1)}$
is the same as the set of column residues of cells containing 
the letter $i$ in $ext(A)$.
Since Definition~\ref{wierdchoice} depends only 
on residues, it remains to show that the index
vector on a standard sequence matches. 

Note that the cocharge index of a standard sequence of $T$ 
can be defined by ordering the residues of $T$ with respect to
\begin{equation}
\label{ord}
x+1<x+2<\cdots <0<1<\cdots <x-1<x \,,\quad
\text{where $x=\lambda_1^{(i)}-1 \pmod n$,}
\end{equation}
and setting $I_i=I_{i-1}$ when the residue of the cell 
containing $i$ is larger than the residue of $i-1$ and $I_i=I_{i-1}+1$ 
otherwise.  Recall also that the index for an $ABC$ is computed
by geographically comparing a cell containing $i-1$ in $ext(A)$ 
to a cell containing $i$; if $i$ is east of $i-1$ then
$I_i = I_{i-1}$ and $I_i=I_{i-1}+1$ otherwise.
We claim that the colum residue of 
the cell containing $i$ is larger than that containing $i-1$
only when $i$ is east of $i-1$.

First we claim that if $\lambda_1^{(i-1)} < \lambda_1^{(i)}$, then there is no ribbon $S$ containing $i-1$ that has
a non-tail cell in columns $[\lambda_1^{(i-1)}+1, \lambda_1^{(i)}]$ 
of $A$. 
By way of contradiction suppose such a ribbon $S$ containing $i-1$ with a non-tail cell $c$ of column residue $r$ exists in $A$.  By Lemma
\ref{lem:abcres}, $\lambda^{(i-1)}$ has a cell of residue $r$.  Since $c$
is in columns $[\lambda_1^{(i-1)}+1, \lambda_1^{(i)}]$, where $\lambda_1^{(i-1)} < \lambda_1^{(i)}$, then $\lambda^{(i)}/\lambda^{(i-1)}$ has
a cell in the bottom row of residue $r$.  Thus $T$ has two distinct cells, not on the same diagonal and of residue $r$.  This contradicts the assumption that $deg(\lambda) < n$.

Observe that the $(i-1)^{st}$ row of $A$ has its rightmost cell in column $\lambda_1^{(i-2)} + n - 1$.  Since $ext(A)$ is constructed by first appending $\lambda_1^{(i-1)} + \lambda_1^{(i-2)} + 1$ cells to the $(i-1)^{st}$ row of $A$, then the cells containing $i-1$ in $ext(A)$ are within columns
$$
[\lambda_1^{(i-1)}+1,(\lambda_1^{(i-2)}+n-1) + (\lambda_1^{(i-1)} - \lambda_1^{(i-2)} + 1)] = [\lambda_1^{(i-1)}+1,\lambda_1^{(i-1)}+n].
$$  
Similarly, the cells containing $i$ in $ext(A)$ are within columns $[\lambda_1^{(i)}+1,\lambda_1^{(i)}+n]$. Since there are no ribbons containing $i-1$ that have a non-tail cell in columns $[\lambda_1^{(i-1)}+1,\lambda_1^{(i)}]$ of $A$, then the cells containing $i-1$ in $ext(A)$ are actually within columns $[\lambda_1^{(i)}+1,\lambda_1^{(i-1)}+n]$.  Since $[\lambda_1^{(i)}+1,\lambda_1^{(i-1)}+n] \subseteq [\lambda_1^{(i)}+1,\lambda_1^{(i)}+n]$, then when $i$ is east of $i-1$ in $ext(A)$, its column residue is larger with respect the same ordering \eqref{ord}.

It remains to prove that $off(A) = 0$.  
By contradiction, if $off(A) > 0$ then $A$ has a ribbon $O$ of 
length greater than 1 and filled with the letter $i$ that is not 
in the $i^{th}$ row, for some $1 \leq i \leq \ell(\mu)$.  
On the one hand, note that $O$ has cells which are at the top of their columns 
in $R(n-1,\lambda^{(i-1)})$, and not in the bottom row.  There is a 
horizontal ribbon in the bottom row of $R(n-1,\lambda^{(i-1)})$ 
whose cells are at the top of their columns and of the same residue 
as those of $O$.  By Lemmas~\ref{lem:abcres}, the set of 
residues labeling cells of $\lambda^{(i)}/\lambda^{(i-1)}$ 
is the same as the set of column residues of cells containing 
letter $i$ in $ext(A)$.  Thus, in $R(n-1,\lambda^{(i-1)})$, 
the residues of non-tail cells of $O$ are the same as the residues 
labeling a horizontal strip $S$ in $\lambda^{(i)}/\lambda^{(i-1)}$.

On the other hand, if $\lambda \in \Co$ with $deg(\lambda) < n$, then no two cells that lie at the top of their columns in $\lambda^{(i)} \subset \lambda$ can have the same $n$-residue.  Thus the residues labeling cells of $\lambda^{(i)}/\lambda^{(i-1)}$ are unique.  Since $O$ is not in the bottom row of $R(n-1,\lambda^{(i-1)})$, and it is of length greater than 1, then it must be the case that there is a non-tail cell of it which is above a cell of $S$.  This contradicts the fact that $R(n-1,\lambda^{(i-1)})/\lambda^{(i-1)}$ is a horizontal strip; constructed by adding a cell to the top of every column of $\lambda^{(i-1)}$ and $n-1$ cells to its bottom row.
\end{proof}

\section{Quantum cohomology of flags}

\label{sec:gw}

Here we show that the strong formulation of the Pieri rule 
for $H_*(\Gr)$ given in Corollary~\ref{thm:Pieristrong}
can be applied to the problem of computing intersections in the 
(small) quantum cohomology of a flag variety.  The examination 
leads to a distinguished family of strong chains defined by a
notion of translation that generalizes $R(n-1,\lambda)$.

In this section, we switch to representing the indices of Schubert 
basis elements by partitions $\lambda \in \mathcal P^n$.
Recall that $R(n-1,\lambda)$ was defined to be $(\lambda_1+n-1,\lambda)$ when
$\lambda$ is an $n$-core in \S~\ref{sec:pieri}.  Here, we abuse notation 
and instead define $R(n-1,\eta)=(\core(\eta)_1+n-1,\core(\eta))$
for $\eta\in\mathcal P^n$. 

\subsection{An affine Monk formula}

The quantum cohomology ring  $QH^*(X)$ is defined for any K\"ahler algebraic 
manifold $X$, but we consider only the complete flag manifold 
$X=Fl_n$ of chains of vector spaces in $\mathbb C^n$.
The quantum cohomolgy ring is simply 
$QH^*(Fl_n)= H^*(Fl_n)\otimes\mathbb Z[q_1,\ldots,q_{n-1}]$
for parameters $q_1,\ldots, q_{n-1}$ as a linear space but
the multiplicative structure is much richer than the specialization 
of $q_1=\cdots=q_{n-1}=0$.
Cells in the Schubert decomposition of $QH^*(Fl_n)$ are indexed 
by permutations $w\in S_n$, and the quantum product is
defined by
\begin{equation}
\label{qschubin}
\sigma_u\,*\,\sigma_v = \sum_w \sum_{d} 
q_1^{d_1} \dots q_{n-1}^{d_{n-1}}\,
 \langle u,v,w\rangle_d
\, \sigma_{w_0w}
\,,
\end{equation}
where the structure constants are 3-point Gromov-Witten invariants 
of genus 0 which count equivalence classes of certain rational curves 
in $Fl_n$.  The understanding and computation of Gromov-Witten invariants
is a widely studied problem.  

Although the construction is not manifestly positive, 
all 3-point, genus zero Gromov-Witten invariants 
$\langle u,w,v\rangle_{d}$ in \eqref{qschubin}
can be computationally obtained from the subset
\begin{equation}
\label{simplegw}
\left\{
\langle s_r,w,v \rangle_d
: 1\leq r<n\;\text{and}\;w,v\in S_n
\right\}
\end{equation}
since the quantum cohomology 
is generated by the codimension one Schubert classes.
By defining a family of {\it quantum Schubert polynomials},
Fomin, Gelfand, and Postnikov were able to prove that
there is a simple combinatorial characterization for 
this set that generalizes the classical Monk formula \cite{Monk}.

\begin{theorem}\cite{FGP}
(Quantum Monk formula)
For $w\in S_n$ and $1\leq r<n$, the
quantum product of
the Schubert classes $\sigma_{s_r}$ and $\sigma_w$ is
given by
\begin{equation}
\label{quantummonk}
\sigma_{s_r}*\sigma_w = 
\sum
\sigma_{w\tau_{a,b}}
+\sum q_{c}q_{c+1}\cdots q_{d-1} \sigma_{w\tau_{c,d}}
\end{equation}
where the first sum is over all transpositions 
$\tau_{a,b}$ such that $a\leq r<b$ and $\ell(w\tau_{a,b})=\ell(w)+1$,
and the second sum is over all transpositions 
$\tau_{c,d}$ such that $c\leq r<d$ and 
$\ell(w\tau_{c,d}) = \ell(w)-\ell(\tau_{c,d})
=\ell(w)-2(d-c)+1$.
\end{theorem}

We have found that the idea of horizontal strong strips extends to 
include combinatorics of the flag Gromov-Witten invariants and the quantum 
Monk rule.  
Peterson asserted that $QH^*(G/P)$ of a flag variety is, up to localization, 
a quotient of 
the homology $H_*(\Gr_G)$ of the affine Grassmannian 
$\Gr_G$ of $G$ (proven in \cite{LStoda}). As a consequence, 
the Gromov-Witten invariants arise as homology Schubert structure 
constants of $H_*(\Gr_G)$.  
The identification of $\langle u,v,w\rangle_d$ 
with the coefficients in
\begin{equation}
\xi_\mu\,\xi_\lambda = \sum_\nu c_{\mu,\lambda}^\nu\,\xi_\nu
\end{equation}
was made explicit in \cite{LM:flag}.  The identification hinges on a
correspondence between permutations in $S_n$ and certain partitions
defined by
\begin{equation}
\sh:
w\mapsto \lambda \quad\text{for}\quad
\lambda_i'=\binom{n-i}{2}+inv_i(w_0w)\,,
\end{equation}
where $\lambda'$ is the partition obtained by
reflecting the shape of $\lambda$ about the line $y=x$,
the inversion $inv_i(u)$ is the number of
$u_j<u_i$ for $i<j$, and $w_0=[n,n-1,\ldots,1]$ is the
permutation of maximal length in $S_n$.

\begin{theorem}\cite{LM:flag}
\label{the:gw2klr}
For $u,v,w\in S_n$ and $d\in\mathbb N^{n-1}$,
\begin{equation}
\label{gw2klr}
\langle u,w,v\rangle_d = 
c_{\sh(u),\sh(w)}^{\eta}
\,,
\end{equation}
where $\eta$ is obtained by
adding $\binom{n+1-i}{2}-(n-i+1)d_i+(n-i)d_{i-1}$
cells to column $i$ of $\sh(v)$, for $1\leq i<n$.
\end{theorem}

The image of $S_n$ under the map $\sh$ is the set of $n!$ partitions 
$\mathcal P^n_\square = \{ \lambda: \square/\lambda =\text{vertical strip}\}$,
where the partition $\square=(n-1,n-2^2,\ldots,1^{n-1})$.  
This foreshadows that \textit{$(n-1)$-rectangles}, 
the shapes $R_r=(r^{n-r})$ with $n-r$ rows of length $r$, 
play a role in the combinatorics of quantum cohomology
of flag varieties as they have in various contexts of affine 
Schubert calculus (e.g. \cite{LLM,[LMfil],Mag,LStoda,BBTZ}).
At the root, for any partition $\lambda \in \mathcal P^n$ and $1\leq r<n$,
\begin{equation}
\label{krec}
\xi_{\lambda\cup R_r}=\xi_{R_r}\xi_{\lambda}
\,.
\end{equation}
Since a defining subset of Gromov-Witten invariants
is given by \eqref{simplegw}, looking closely at $\sh(s_r)$ 
reveals the role of these shapes in the combinatorics of 
quantum cohomology.  To be precise, it was shown in \cite{LM:flag} that
for any $v,w\in S_n$ and $1\leq r<n$, 
\begin{equation}
\label{thm:simpleklr}
\langle s_r,v,w\rangle_d= c_{R_r',\sh(w)}^{\eta\cup\sh(v)\cup R_r}\,,
\end{equation}
where the $i$th column of $\eta$ is
$(n-i)d_{i-1}-(n+1-i)d_i$ and
$R'_r$ is the shape obtained by deleting the corner box from $R_r$.
In particular, Monk's classical formula is determined by
$c_{R_r',\sh(w)}^{\sh(v)\cup R_r}$.
Therefore, for any $u,v,w\in S_n$, all Gromov-Witten invariants 
$\langle u,v,w\rangle_d$ can be computed from
the set
\begin{equation}
\label{simplelr}
\left\{c_{R_r',\lambda}^{\nu} :
1\leq r<n,\, \lambda\in\mathcal P^n_\square,\,\text{and}\, \nu_1<n
\right\}\,.
\end{equation}

The $\eta\cup \sh(v)\cup R_r$ in \eqref{thm:simpleklr} 
suggests that the formula for these invariants is 
related to elements covered by the generic translation of
$\lambda$ defined by 
$$R(r,\lambda)=\core(\lambda\cup R_r)\,.
$$

%

\begin{conjecture}
\label{con:amf}
(Affine Monk formula)
For $1\leq r<n$ and partition  $\lambda$ with $\lambda_1<n$,
\begin{equation}
\label{affmonk}
\xi_{R_r'} \,\xi_{\lambda} = 
\sum_{\core(\nu)\lessdot_B \,R(r,\lambda)}
\xi_{\nu} \,,
\end{equation}
where $\core(\nu)_i < R(r,\lambda)_i$ for
some $i$ such that $(\lambda\cup R_r)_i=r$.
\end{conjecture}
Note that the expansion \eqref{affmonk} can be derived from results in
\cite{BSS} that determine the expansion of a \textit{non-commutative $k$-Schur function} 
indexed by $R_r'$ in terms of words in the affine nilCoxeter algebra $\mathbb{A}$.

\begin{example}
\label{ex:affmonk}
For $n=5$, $\lambda=(3,2,1,1)$, and $R_3'=(3,2)$, 
term $\xi_\nu$ occurs in the expansion of $\xi_{R_3'}\xi_\lambda$ 
when $\nu$ is a partition where $\mfc(\nu) \lessdot_B R(3,\lambda)$ and 
$\mfc(\nu)_i < R(3,\lambda)_i$ for some $i \in \{1,2,3\}$.
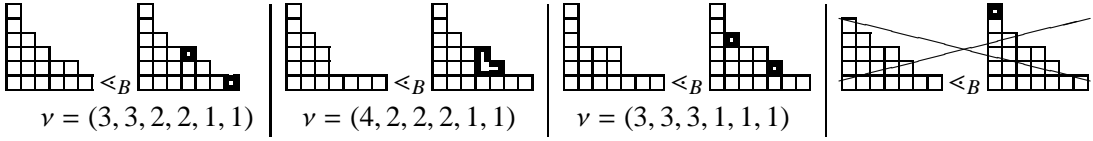
\bigskip

\begin{tabular}{c|c|c|c}
\hspace{-0.3in}
${\text{\Large\tableau*[pbY]{\cr \cr & \cr & & \cr & & & & \cr & & & & &}}} 
\lessdot_B
{\text{\Large\tableau*[pbY]
{ &  \bl &  \bl &  \bl \cr \cr & \cr & & & \tf \cr & & & & \cr & & & & & & \tf }}}$
&
${\text{\Large\tableau*[pbY]{\cr  \cr & \cr & & \cr & & \cr & & & & & & }}} 
\lessdot_B
{\text{\Large\tableau*[pbY]
{ &  \bl &  \bl &  \bl \cr \cr & \cr & & & \fr[w,l,t,r] \cr & & & \fr[w,l,b] & \fr[w,t,b,r] \cr & & & & & & }}}$
&
${\text{\Large\tableau*[pbY]{\cr \cr \cr & & & \cr & & & \cr & & & & & &}}} 
\lessdot_B
{\text{\Large\tableau*[pbY]
{ &  \bl &  \bl &  \bl \cr \cr & \tf \cr & & & \cr & & & & \tf \cr & & & & & & }}}$
&
$\xcancel{{\text{\Large\tableau*[pbY]{ \cr & \cr & & & \cr & & & & \cr & & & & & &}}} 
\lessdot_B {\text{\Large\tableau*[pbY]
{ \tf &  \bl &  \bl & 
 \bl \cr 
\cr & \cr & & &\cr & & & & \cr & & & & & &}}} }$	\\
$\nu=(3,3,2,2,1,1)$\quad & $\nu=(4,2,2,2,1,1)$\quad  & $\nu=(3,3,3,1,1,1)$ \quad &
\end{tabular}

\medskip

\noindent
Relation~\eqref{krec} then gives the terms in the expansion of 
$\xi_{R_3'}\xi_{\lambda\cup R_r}$ for any $R_r$.
In particular, 
\label{ribbonstripsconjex}
\begin{equation}
\label{kschurprod}
\xi_{R_3'}\,\xi_{\lambda\cup R_2} = 
\xi_{(3,3,2,2,1,1)\cup R_2}
+ \xi_{(4,2,2,2,1,1)\cup R_2}
+ \xi_{(3,3,3,1,1,1)\cup R_2}\,.
\end{equation}
Since $\sh([4,2,5,3,1])=\lambda\cup R_2\in \mathcal P_\square^n$,
this matches the quantum Monk expansion
by Equation~\eqref{thm:simpleklr}:
$$
\sigma_{s_3}\,*\,\sigma_{[4,2,5,3,1]} = 
\sigma_{[4,3,5,2,1]} +
q_3 \,\sigma_{[4,2,3,5,1]} +
q_3q_4\,\sigma_{[4,2,1,3,5]}
\,.
$$
%
%
%
%
\end{example}

\subsection{Ribbon strong strips}

When $r=n-1$, Conjecture~\ref{con:amf} reduces to a special case 
of the expansion given in Corollary~\ref{thm:Pieristrong}.  The
terms are defined by horizontal strong 1-strips, which are in fact
the elements $\nu$ covered by $R(n-1,\lambda)$ where 
$\core(\nu)_1<R(n-1,\lambda)_1$.
Horizontal strong strips of generic length $1\leq b<n-1$
describe the expansion of $\xi_{(n-1-b)}\xi_\lambda$.
We thus turn to the more general expansion,
for $1\leq b<r<n$ and partition  $\lambda \in \mathcal P^n$,
\begin{equation}
\label{rectconj}
\xi_{(r^{n-1-r},r-b)} \,\xi_\lambda = \sum_{\nu\in \mathcal B_{r,b,\lambda}}
\xi_\nu
\,,
\end{equation}
as a guide to characterize a larger family of
strong strips associated to the general translation $R(r,\lambda)$.

While horizontal strong strips are certain shapes differing 
from $R(n-1,\lambda)$ by a horizontal strip, the more general picture 
involves shapes that differ from  $R(r,\lambda)$
by a {\it horizontal ribbon strip}, a
sequences of shapes
$\nu= \nu^{(0)}\subset\nu^{(1)}\subset\cdots\subset\nu^{(m)}$
such that $\nu^{(i)}/\nu^{(i-1)}$ is comprised of ribbons
whose heads lie above a cell in $\nu$ (or in the bottom row),
for all $1\leq i\leq m$.  Rather than requiring that bottom
row lengths are increasing as we did for horizontal strong 
strip, we now require that a ribbon tail lies in a specified set 
of columns.  
For $1\leq r<n$ and a partition $\lambda$ with $\lambda_1<n$,
let $\eta=\lambda\cup R_r$.  
Let $m$ be the highest row of $\eta$ that has length $r$
and denote the set of $r$ columns containing the
last $r$ cells in row $m$ of $\core(\eta)$ by
$col_r(\lambda)$.


\begin{definition}
\label{rstrongstrip}
Given $n$-cores $\lambda$ and $\nu$ and $1\leq r<n$,
the pair $(\lambda,\nu)$ is a ribbon strong strip 
with respect to $r$ if there is a horizontal ribbon
strip
$$
\nu=\nu^{(0)}\lessdot_B\nu^{(1)}\lessdot_B\cdots\lessdot_B\nu^{(m)}=
R(r,\lambda)
$$
with a ribbon tail of $\nu^{(i)}/\nu^{(i-1)}$ 
lying in $col_r(\lambda)$ for all $i>0$.
Its length is defined to be $m$.
\end{definition}
The expansion \eqref{rectconj} can be derived from results in
\cite{BSSgen} that determine the expansion of a \textit{non-commutative $k$-Schur function} 
indexed by $(r^{n-1-r},b)$ in terms of words in the affine nilCoxeter algebra $\mathbb{A}$.
We instead conjecture that the expansion is simply the sum over $\nu$ 
such that $(\core(\lambda),\core(\nu))$ 
is a ribbon strong strip of length $b$ with respect to $r$.

\begin{example}
\label{ex:rect}
For $n=5$ and $\lambda = (4,2)$, the ribbon strong strips 
of length 2 with respect to $r=3$ are
\vspace{-0.03in}
$$
\begin{tabular}{c|c}
{\Large\tableau*[pbY]{ \cr \cr & & & & \cr & & & & & & & &}} $\lessdot_B$ 
{\Large\tableau*[pbY]{ \bl {\color{red}{\downarrow}} & \bl {\color{red}{\downarrow}} & \bl {\color{red}{\downarrow}} \cr \bl \cr & \fr[w,l,t,r] \cr & \fr[w,l,b,r] \cr & & & & \cr & & & & & & & &}} $\lessdot_B$ 
{\Large\tableau*[pbY]{ \bl {\color{red}{\downarrow}} & \bl {\color{red}{\downarrow}} & \bl {\color{red}{\downarrow}} \cr \bl \cr & \cr & & \tf \cr & & & & \cr & & & & & & & &}}
& 
{\Large\tableau*[pbY]{ \cr & \cr & & & \cr & & & & & & &}} $\lessdot_B$ 
{\Large\tableau*[pbY]{ \bl {\color{red}{\downarrow}} & \bl {\color{red}{\downarrow}} & \bl {\color{red}{\downarrow}} \cr \bl \cr & \tf \cr & \cr & & & & \tf \cr & & & & & & & &\tf }} $\lessdot_B$ 
{\Large\tableau*[pbY]{ \bl {\color{red}{\downarrow}} & \bl {\color{red}{\downarrow}} & \bl {\color{red}{\downarrow}} \cr \bl \cr & \cr & & \tf \cr & & & & \cr & & & & & & & &}} \\ \hline \\[-0.08in] 
{\Large\tableau*[pbY]{ & & \cr & & \cr & & & & & &}} $\lessdot_B$ 
{\Large\tableau*[pbY]{ \bl {\color{red}{\downarrow}} & \bl {\color{red}{\downarrow}} & \bl {\color{red}{\downarrow}} \cr \bl \cr \tf \cr & & \cr & & & \tf \cr & & & & & & & \tf}} $\lessdot_B$ 
{\Large\tableau*[pbY]{ \bl {\color{red}{\downarrow}} & \bl {\color{red}{\downarrow}} & \bl {\color{red}{\downarrow}} \cr \bl \cr & \tf \cr & & \cr & & & & \tf \cr & & & & & & & & \tf}}
&
{\Large\tableau*[pbY]{ \cr & \cr & & & \cr & & & & & & &}} $\lessdot_B$ 
{\Large\tableau*[pbY]{ \bl {\color{red}{\downarrow}} & \bl {\color{red}{\downarrow}} & \bl {\color{red}{\downarrow}} \cr \bl \cr \cr & & \tf \cr & & & \cr & & & & & & &}} $\lessdot_B$ 
{\Large\tableau*[pbY]{ \bl {\color{red}{\downarrow}} & \bl {\color{red}{\downarrow}} & \bl {\color{red}{\downarrow}} \cr \bl \cr & \tf \cr & & \cr & & & & \tf \cr & & & & & & & & \tf}}
\end{tabular}
$$
Conjecturally, this gives the expansion
$$
\xi_{{\text{\large\tableau*[pbY]{ \cr & &}}}} \xi_{{\text{\large\tableau*[pbY]{ & \cr & & &}}}} = \xi_{{\text{\large\tableau*[pbY]{ \cr \cr & & & \cr & & & }}}} + \xi_{{\text{\large\tableau*[pbY]{ & & \cr & & \cr & & & }}}} + \xi_{{\text{\large\tableau*[pbY]{ \cr & \cr & & \cr & & & }}}}.
$$
\end{example}

\begin{conjecture}
Given $n$-cores $\lambda$ and $\nu$ and $1\leq r<n$,
the pair $(\lambda,\nu)$ is a ribbon strong strip 
with respect to $r$ if and only if
there exists a strong strip 
\begin{equation}
\label{srchain}
\nu=\nu^{(0)}\lessdot_B\nu^{(1)}
\lessdot_B \cdots \lessdot_B \nu^{(m)}= R(r,\lambda)
\end{equation}
with a ribbon tail of $\nu^{(i)}/\nu^{(i-1)}$ 
lying in $col_r(\lambda)$ for all $i>0$.
\end{conjecture}

\begin{proposition}
Given $n$-cores $\lambda$ and $\nu$,
$(\lambda,\nu)$ is a horizontal strong strip
if and only if $(\lambda,\nu)$ is a ribbon strong strip 
with respect to $n-1$.
\end{proposition}
\begin{proof}
Let $p$ be the number of rows of length $n-1$ in $\kbnd(\lambda)$
and thus the bottom $p+1$ rows of $\kbnd(\lambda)\cup R_{n-1}$ have 
length $n-1$.  Note by definition of $\core$ 
that the last $n-1$ cells in rows $1,\ldots,p+1$ of
$\core(\kbnd(\lambda)\cup R_{n-1})$ lie at the top
of a column and 
have residues $\lambda_1,\lambda_1+1,\ldots,\lambda_1+(n-2)$ (mod $n$).  
Further, $col_{n-1}(\lambda)$ is defined by taking the 
last $n-1$ columns in row $p+1$.

Assume $(\lambda,\nu)$ is a horizontal strong strip.
Lemma~\ref{lem:hstrip} implies that $\nu^{(j)}/\nu^{(j-1)}$ consists 
of all copies of $S^j$ that can be removed from $\nu^{(j)}$ where
$S^j$ is a removable ribbon in the bottom row of $\nu^{(j)}$.
For $j=m$, the discussion in the previous paragraph
implies that a copy of $S^m$ lies in row $p+1$.
By iteration, a copy of $S^j$ (and in particular, its tail)
lies in the last $n-1$ columns of row $p+1$ for all $j=1,\ldots,m$.

On the other hand, consider a chain of $n$-cores
$\nu=\nu^{(0)}\lessdot_B\nu^{(1)}\lessdot_B\cdots\lessdot_B\nu^{(m)}=
R(n-1,\lambda)$ where the tail of a ribbon $S^j$ in $\nu^{(j)}/\nu^{(j-1)}$ 
lies in one of the last $n-1$ columns in row $p+1$ of $R(n-1,\lambda)$.
Since the number of cell in $S^j$ is smaller than $n$ and
there are $n-1$ cells at the top of a column in row $p$,
$S^m$ must have height one.  Therefore, it can be removed from
every row $1,\ldots,p+1$ of $R(n-1,\lambda)$.
By iteration, there is a copy of $S^j$ in the bottom row of 
$(n-1+\lambda_1,\lambda)$ for $j=1,\ldots,m$.
Since the tail of $S^1$ is on of the last $n-1$ cells
in row $p+1$ of residue $\lambda_1,\ldots,\lambda_1+(n-2)\mod n$,
$\lambda\subset\nu$.
\end{proof}

\begin{proposition}
For $1\leq r<n$ and for $n$-cores $\lambda$ and $\nu$,
a ribbon strong strip $(\lambda,\nu)$ with respect to $r$
has length one 
if and only if $\nu\lessdot_B R(r,\lambda)$ and 
$\nu_i< R(r,\lambda)_i$ for some $i$ such that
$(\lambda\cup R_r)_i=r$.
\end{proposition}
\begin{proof}
For $\eta\in\mathcal P^n$, let $\lambda=\core(\eta)$ and 
let $m$ be the highest row of length $r$ in $R_r\cup \eta$.
Note that the highest cell of $R(r,\lambda)$ in the leftmost 
column of $col_r(\lambda)$ lies in a row no higher than row 
$m+(n-1-r)$ and that rows $m,m-1,\ldots,m-(n-1-r)$ of 
$R_r\cup \eta$ have length $r$.

Given $(\lambda,\nu)$ is a ribbon strong strip with respect to $r$
of length one, $\nu\lessdot_B R(r,\lambda)$ and the tail of a
ribbon $S\subset R(r,\lambda)/\nu$ lies in a column of $col_r(\lambda)$.
Suppose the head of $S$ lies in row $a$.  If $a\leq m$,
then $\nu\lessdot_B R(r,\lambda)$ and 
$\nu_i< R(r,\lambda)_i$ for some $i$ such that
$(\lambda\cup R_r)_i=r$.  When $a>m$, all cells of $S$ 
lie in colums of $col_r(\lambda)$ and therefore in rows
between $m+1$ and $m+n-1-r$.  Proposition 9~\cite{[LMrec]} 
ensures that a removable copy of $S$ also lies in rows 
$m-(n-1-r),\ldots,m-1,m$, and the forward direction thus follows 
from Lemma~\ref{lem:tau}.

On the other hand, consider $n$-cores $\lambda$ and $\nu$ such that
$\nu_i \lessdot_B R(r,\lambda)_i$ for some $i$ where $(\lambda \cup R_r)_i = r$.  
In particular, there is a ribbon $S\subset R(r,\lambda)/\nu$ containing 
at least one cell in row $i$.  If the tail $t$ of $S$ is not in a column 
of $col_r(\lambda)$ then $t$ lies in row $i<m$.  By Proposition 9~\cite{[LMrec]}, 
there is an extremal cell $\tilde{t}$ of the same residue as $t$
that lies in a column of $col_r(\lambda)$.  Consider the subset $\tilde S$ 
of extremal cells in $R(r,\lambda)$ that is formed by taking 
all extremal cells between $\tilde t$ (as the highest) and a cell
$\tilde h$ of the same residue as the head of $S$.
Since $S$ can be removed from $R(r,\lambda)$, its head
lies at the end of its row. Property~\ref{thecoreprop} then implies
that $\tilde h$ is at the end of its row and thus
$\tilde S$ is a removable ribbon. By Lemma~\ref{lem:tau},
$\tilde S\subset R(r,\lambda)/\nu$ proving the claim.
\end{proof}

\bibliographystyle{alpha}
\bibliography{avi}
\label{sec:biblio}

\end{document}